\documentclass[12pt]{amsart}

\usepackage[T2A]{fontenc}
\usepackage{amsfonts}
\usepackage{amsbsy}
\usepackage{amssymb}
\usepackage{mathrsfs}
\usepackage[english]{babel}
\usepackage[cp1251]{inputenc}
\usepackage{xypic}
\usepackage{hyperref}
\usepackage{enumitem}

\renewcommand{\abovecaptionskip}{0pt}
\renewcommand{\belowcaptionskip}{6pt}
\makeatletter

\@addtoreset{theorem}{section}

\renewcommand{\@makecaption}[2]{
\vspace{\abovecaptionskip}%
\sbox{\@tempboxa}{#1 #2}%
\global\@minipagefalse \hbox to \hsize {{\scshape \hfil #1 #2\hfil}}
\vspace{\belowcaptionskip}}

\makeatother

\renewcommand{\ge}{\geqslant}
\renewcommand{\le}{\leqslant}

\newcommand{\rk}{\operatorname{rk}}

\newcommand{\SL}{\operatorname{SL}}
\newcommand{\GL}{\operatorname{GL}}
\newcommand{\Sp}{\operatorname{Sp}}
\newcommand{\Spin}{\operatorname{Spin}}
\newcommand{\SO}{\operatorname{SO}}

\newcommand{\Spec}{\operatorname{Spec}}

\newcommand{\FF}{\mathbb F}
\newcommand{\PP}{\mathbb P}

\newcommand{\reg}{\mathrm{reg}}
\newcommand{\bl}{\mbox{$[\hspace{-0.35ex}[$}}
\newcommand{\br}{\mbox{$]\hspace{-0.35ex}]$}}

\DeclareMathOperator{\Aut}{\mathrm{Aut}}
\DeclareMathOperator{\Lie}{\mathrm{Lie}}
\DeclareMathOperator{\Fl}{\mathrm{Fl}}
\DeclareMathOperator{\YD}{\mathrm{YD}}
\DeclareMathOperator{\Gr}{\mathrm{Gr}}
\DeclareMathOperator{\Quot}{\mathrm{Quot}}
\DeclareMathOperator{\Hom}{\mathrm{Hom}}

\numberwithin{equation}{section}

\newtheorem{theorem}{Theorem}
\newtheorem{proposition}[theorem]{Proposition}
\newtheorem{lemma}[theorem]{Lemma}
\newtheorem{corollary}[theorem]{Corollary}
\newtheorem{problem}[theorem]{Problem}

\theoremstyle{definition}

\newtheorem{definition}[theorem]{Definition}

\theoremstyle{remark}

\newtheorem{remark}[theorem]{Remark}

\newcommand{\dd}{\begin{picture}(6,6)\put(0,6){\circle*{1.1}}
\put(3,3){\circle*{1.1}}\put(6,0){\circle*{1.1}}\end{picture}}
\newcommand{\vd}{\begin{picture}(0,6)\put(0,6){\circle*{1.1}}
\put(0,3){\circle*{1.1}}\put(0,0){\circle*{1.1}}\end{picture}}

\begin{document}

\sloppy

\renewcommand{\proofname}{Proof}
\renewcommand{\abstractname}{Abstract}
\renewcommand{\refname}{References}
\renewcommand{\figurename}{Figure}
\renewcommand{\tablename}{Table}

\title[Spherical actions on flag varieties]%
{Spherical actions on flag varieties}

\author{Roman Avdeev and Alexey Petukhov}

\thanks{The first author was supported by the RFBR grant
no.~12-01-00704, the SFB 701 grant of the University of Bielefeld
(in 2013), the ``Oberwolfach Leibniz Fellows'' programme of the
Mathematisches Forschungsinstitut Oberwolfach (in 2013), and Dmitry
Zimin's ``Dynasty'' Foundation (in 2014). The second author was
supported by the Guest Programme of the Max-Planck Institute for
Mathematics in Bonn (in 2012--2013).}

\address{%
{\bf Roman Avdeev} \newline\indent National Research University
``Higher School of Economics'', Moscow, Russia
\newline\indent Moscow Institute of Open
Education, Moscow, Russia}

\email{suselr@yandex.ru}

\address{%
{\bf Alexey Petukhov}%
\newline\indent Institute for Information Transmission Problems,
Moscow, Russia}

\email{alex{-}{-}2@yandex.ru}




\begin{abstract}
For every finite-dimensional vector space $V$ and every $V$-flag
variety $X$ we list all connected reductive subgroups in~$\GL(V)$
acting spherically on~$X$.
\end{abstract}

\maketitle


\section{Introduction}

Throughout this paper we fix an algebraically closed field $\mathbb
F$ of characteristic~$0$, which is the ground field for all objects
under consideration. By $\FF^\times$ we denote the multiplicative
group of~$\FF$.

Let $K$ be a connected reductive algebraic group and let $X$ be a
$K$-variety (that is, an algebraic variety together with a regular
action of~$K$). The action of~$K$ on~$X$, as well as the variety $X$
itself, is said to be \textit{spherical} (or \textit{$K$-spherical}
when one needs to emphasize the acting group) if $X$ is irreducible
and a Borel subgroup of~$K$ has an open orbit in~$X$. Every
finite-dimensional $K$-module that is spherical as a $K$-variety is
said to be a \textit{spherical $K$-module}. Spherical varieties
possess various remarkable properties; a review of them can be
found, for instance, in the monograph by Timashev~\cite{Tim}.

Describing all spherical varieties for a fixed group $K$ is an
important and interesting problem, and by now considerable results
have been achieved in solving it. A review of these results can be
also found in the monograph~\cite{Tim}. Along with the problem
mentioned above one can also consider an opposite one, namely, for a
given algebraic variety $X$ find all connected reductive subgroups
in the automorphism group of~$X$ that act spherically on~$X$. In
this paper we consider this problem in the case where $X$ is a
generalized flag variety.

A \textit{generalized flag variety} is a homogeneous space of the
form $G / P$, where $G$ is a connected reductive group and $P$ is a
parabolic subgroup of~$G$. We call the variety $G / P$
\textit{trivial} if $P = G$ and \textit{nontrivial} otherwise. The
following facts are well known:

(1) all generalized flag varieties of a fixed group~$G$ are exactly
all complete (and also all projective) homogeneous spaces of~$G$;

(2) the center of $G$ acts trivially on~$G / P$;

(3) the natural action of $G$ on $G/P$ is spherical.

Let $\mathscr F(G)$ denote the set of all (up to a $G$-equivariant
isomorphism) nontrivial generalized flag varieties of a fixed
group~$G$. For every generalized flag variety~$X$ we denote by $\Aut
X$ its automorphism group.

In~\cite[Theorem~7.1]{Oni} (see also~\cite{Dem}) the automorphism
groups of all generalized flag varieties were described. This
description implies that, for every $X \in\nobreak \mathscr F(G)$,
$\Aut X$ is an affine algebraic group and its connected component of
the identity $(\Aut X)^0$ is semisimple and has trivial center.
Moreover, it turns out that in most of the cases (all exceptions are
listed in~\cite[Table~8]{Oni}; see also~\cite[\S\,2]{Dem}),
including the case $G = \GL_n$, the natural homomorphism $G \to
(\Aut X)^0$ is surjective and its kernel coincides with the center
of~$G$. In any case one has $X \in \mathscr F((\Aut X)^0)$,
therefore the problem of describing all spherical actions on
generalized flag varieties reduces to the following one:

\begin{problem} \label{prob_main}
For every connected reductive algebraic group $G$ and every variety
$X \in \mathscr F(G)$, list all connected reductive subgroups $K
\subset G$ acting spherically on~$X$.
\end{problem}

Since the center of $G$ acts trivially on any variety $X \in
\mathscr F(G)$, in solving Problem~\ref{prob_main} the group~$G$
without loss of generality may be assumed semisimple when it is
convenient.

The main goal of this paper is to solve Problem~\ref{prob_main} in
the case $G = \GL_n$ (the results are stated below in
Theorems~\ref{thm_sphericity_equivalence}
and~\ref{thm_main_theorem}).

Let us list all cases known to the authors where
Problem~\ref{prob_main} is solved.

1) $G$ and $X$ are arbitrary and $K$ is a Levi subgroup of a
parabolic subgroup $Q \subset G$. In this case the condition that
the variety $G / P \in \mathscr F(G)$ be $K$-spherical is equivalent
to the condition that the variety $G / P \times G / Q$ be
$G$-spherical, where $G$ acts diagonally (see
Lemma~\ref{lemma_two_parabolics}). All $G$-spherical varieties of
the form $G / P \times G / Q$ have been classified. In the case
where $P,Q$ are maximal parabolic subgroups this was done by
Littelmann in~\cite{Lit}. In the cases $G = \GL_n$ and $G =
\Sp_{2n}$ the classification follows from results of the Magyar,
Weyman, and Zelevinsky papers~\cite{MWZ1} and~\cite{MWZ2},
respectively. (In fact, in~\cite{MWZ1} and~\cite{MWZ2} for $G =
\GL_n$ and $G = \Sp_n$, respectively, the following more general
problem was solved: describe all sets $X_1, \ldots, X_k \in \mathscr
F(G)$ such that~$G$ has finitely many orbits under the diagonal
action on $X_1 \times \ldots \times X_k$.) Finally, for arbitrary
groups~$G$ the classification was completed by Stembridge
in~\cite{Stem}. The results of this classification for $G = \GL_n$
will be essentially used in this paper and are presented in
\S\,\ref{subsec_Levi_subgroups}.

2) $G$ and $X$ are arbitrary and $K$ is a symmetric subgroup of~$G$
(that is, $K$ is the subgroup of fixed points of a nontrivial
involutive automorphism of~$G$). In this case the classification was
obtained in the paper~\cite{HNOO}.

3) $G$ is an exceptional simple group, $X = G / P$ for a maximal
parabolic subgroup $P \subset G$, and $K$ is a maximal reductive
subgroup of~$G$. This case was investigated in the
preprint~\cite{Nie}.

Let $G$ be an arbitrary connected reductive group, $\mathfrak g =
\Lie G$, and $K \subset G$ an arbitrary connected reductive
subgroup.

We now discuss the key idea utilized in this paper. Let $P
\subset\nobreak G$ be a parabolic subgroup and let $N \subset P$ be
its unipotent radical. Put $\mathfrak n = \Lie N$. Consider the
adjoint action of~$G$ on~$\mathfrak g$. In view of a well-known
result of Richardson (see~\cite[Proposition~6(c)]{Rich74}), for the
induced action $P : \mathfrak n$ there is an open orbit $O_P$. We
put
\begin{equation} \label{eqn_nilp_orbit}
\mathcal N(G / P) = G O_P \subset \mathfrak g.
\end{equation}
It is easy to see that $\mathcal N(G / P)$ is a nilpotent (that is,
containing $0$ in its closure) $G$-orbit in~$\mathfrak g$.

\begin{definition}
We say that two varieties $X_1, X_2 \in \mathscr F(G)$ are
\textit{nil-equivalent} (notation: $X_1 \sim\nobreak X_2$) if
$\mathcal N(X_1) = \mathcal N(X_2)$.
\end{definition}

It is well-known that every $G$-orbit in $\mathfrak g$ is endowed
with the canonical structure of a symplectic variety. It turns out
(see Theorem~\ref{thm_spherical&coisotropic_orbit}) that a variety
$X \in \mathscr F(G)$ is $K$-spherical if and only if the action $K
: \mathcal N(X)$ is coisotropic (see
Definition~\ref{dfn_coisotropic}) with respect to the symplectic
structure on $\mathcal N(X)$. This immediately implies the following
result.

\begin{theorem} \label{thm_equiv_spherical}
Suppose that $X_1, X_2 \in \mathscr F(G)$ and $X_1 \sim X_2$. Then
the following conditions are equivalent:

\textup{(a)} the action $K : X_1$ is spherical;

\textup{(b)} the action $K : X_2$ is spherical.
\end{theorem}

Let $\bl X \br$ denote the nil-equivalence class of a variety $X \in
\mathscr F(G)$. The inclusion relation between closures of nilpotent
orbits in~$\mathfrak g$ defines a partial order $\preccurlyeq$ on
the set $\mathscr F(G) /\! \sim$ of all nil-equivalence classes in
the following way: for $X_1, X_2 \in \mathscr F(G)$ the relation
$\bl X_1 \br \preccurlyeq\nobreak \bl X_2 \br$ (or $\bl X_2 \br
\succcurlyeq \bl X_1 \br$) holds if and only if the orbit $\mathcal
N(X_1)$ is contained in the closure of the orbit $\mathcal N(X_2)$.
We shall also write $\bl X_1 \br \prec \bl X_2 \br$ (or $\bl X_2 \br
\succ\nobreak \bl X_1 \br$) when $\bl X_1 \br \preccurlyeq \bl X_2
\br$ but $\bl X_1 \br \ne \bl X_2 \br$.

Using a result from the resent paper~\cite{Los} by Losev, in
\S\,\ref{subsec_proof_of_thm_prec} we shall prove the following
theorem.

\begin{theorem} \label{thm_prec_spherical}
Suppose that $X_1, X_2 \in \mathscr F(G)$ and $\bl X_1 \br \prec \bl
X_2 \br$. If the action $K : X_2$ is spherical, then so is the
action $K : X_1$.
\end{theorem}

Theorems~\ref{thm_equiv_spherical} and~\ref{thm_prec_spherical}
yield the following method for solving Problem~\ref{prob_main}: at
the first step, for each class $\bl X \br \in \mathscr F(G) / \!
\sim$ that is a minimal element with respect to the
order~$\preccurlyeq$, find all connected reductive subgroups~$K
\subset\nobreak G$ acting spherically on~$X$; at the second step,
using the lists of subgroups $K$ obtained at the first step, carry
out the same procedure for all nil-equivalence classes that are on
the ``next level'' with respect to the order~$\preccurlyeq$; and so
on.

We recall that there is the natural partial order $\le$ on the set
$\mathscr F(G)$. This order can be defined as follows. Fix a Borel
subgroup $B \subset G$. Let $X_1, X_2 \in \mathscr F(G)$. Then $X_1
= G / P_1$ and $X_2 = G / P_2$ where $P_1, P_2$ are uniquely
determined parabolic subgroups of~$G$ containing~$B$. By definition,
the relation $X_1 \le X_2$ (resp. $X_1 < X_2$) holds if and only if
$P_1\supset P_2$ (resp. $P_1\supsetneq P_2$). It is easy to show
that the following analogue of Theorem~\ref{thm_prec_spherical} is
valid for the partial order~$\le$: if $X_1 < X_2$ and the action $K
: X_2$ is spherical, then so is the action $K : X_1$. Now let $N_i$
be the unipotent radical of~$P_i$, $i = 1,2$. Then the condition
$P_1 \supsetneq P_2$ implies that $N_1 \subsetneq N_2$, from which
one easily deduces that the orbit $\mathcal N(X_1)$ is contained in
the closure of the orbit $\mathcal N(X_2)$. As $\dim \mathcal N(X_i)
= 2\dim N_i$ for $i=1,2$ (see, for
instance,~\cite[Theorem~7.1.1]{CM}), one has $\mathcal N(X_1) \ne
\mathcal N(X_2)$. Thus, if $X_1 < X_2$ then $\bl X_1\br \prec \bl
X_2\br$. In particular, if $\bl X \br$ is a minimal element of the
set $\mathscr F(G) / \!\sim$ with respect to the partial
order~$\preccurlyeq$, then $X$ is a minimal element of the set
$\mathscr F(G)$ with respect to the partial order~$\le$.

We note that for a semisimple group $G$ of rank~$n$ the set
$\mathscr F(G)$ contains exactly $n$ minimal elements with respect
to the partial order~$\le$. On the other hand, using known results
on nilpotent orbits in the classical Lie algebras (see~\cite[\S\S
5--7]{CM}), one can show that for a simple group~$G$ of type
$\mathsf A_n$, $\mathsf B_n$, or $\mathsf C_n$ the set $\mathscr
F(G) / \! \sim$ contains only one minimal element with respect to
the partial order~$\preccurlyeq$ and for a simple group~$G$ of
type~$\mathsf D_n$ ($n \ge 4$) the set $\mathscr F(G) / \! \sim$
contains two (for $n = 2k+1$) or three (for $n = 2k$) minimal
elements with respect to the partial order~$\preccurlyeq$. This
shows that the partial order $\preccurlyeq$ turns out to be much
more effective in solving Problem~\ref{prob_main} than the partial
order~$\le$.

We now turn to $V$-flag varieties, which are the main objects in our
paper. In order to fix the definition of these varieties that is
convenient for us, we need the following notion. A
\textit{composition} of a positive integer~$d$ is a tuple of
positive integers $(a_1, \ldots, a_s)$ satisfying the condition
$$
a_1 + \ldots + a_s = d.
$$
We say that a composition $\mathbf a = (a_1, \ldots, a_s)$ is
\textit{trivial} if $s = 1$ and \textit{nontrivial} if $s \ge 2$.

Let $V$ be a finite-dimensional vector space of dimension~$d$ and
let $\mathbf{a} = (a_1, \ldots, a_s)$ be a composition of~$d$. The
\textit{$V$-flag variety} (or simply the \textit{flag variety})
corresponding to~$\mathbf a$ is the set of tuples $(V_1, \ldots,
V_s)$, where $V_1, \ldots, V_s$ are subspaces of~$V$ satisfying the
following conditions:

(a) $V_1 \subset \ldots \subset V_s = V$;

(b) $\dim V_i / V_{i-1} = a_i$ for $i = 1, \ldots, s$ (here we
suppose that $V_0 = \lbrace 0 \rbrace$).

We note that $\dim V_i = a_1 + \ldots + a_i$ for all $i = 1, \ldots,
s$.

For every composition $\mathbf{a} = (a_1, \ldots, a_s)$, we denote
by $\Fl_{\mathbf{a}}(V)$ the $V$-flag variety corresponding
to~$\mathbf a$. If~$\mathbf{a}$ is nontrivial, then along with
$\Fl_{\mathbf a}(V)$ we shall also use the notation $\Fl(a_1,
\ldots, a_{s-1}; V)$.

The following facts are well known:

(1) every $V$-flag variety $X$ is equipped with a canonical
structure of a projective algebraic variety and the natural action
of $\GL(V)$ on $X$ is regular and transitive;

(2) up to a $\GL(V)$-equivariant isomorphism, all $V$-flag varieties
are exactly all generalized flag varieties of the group $\GL(V)$.

In view of fact~(2), in what follows all notions and notation
introduced for generalized flag varieties will be also used for
$V$-flag varieties.

It is easy to see that a $V$-flag variety $\Fl_{\mathbf a}(V)$ is
trivial (resp. nontrivial) if and only if the composition~$\mathbf
a$ is trivial (resp. nontrivial).

For $1 \le k \le d$ we consider the composition $\mathbf c_k$
of~$d$, where $\mathbf c_k = (k, d - k)$ for $1 \le k \le d-1$ and
$\mathbf c_d = (d)$. The variety $\Fl_{\mathbf c_k}(V)$ is said to
be a \textit{Grassmannian}, we shall use the special notation
$\Gr_k(V)$ for it. Points of this variety are in one-to-one
correspondence with $k$-dimensional subspaces of~$V$. The point
corresponding to a subspace $W \subset V$ will be denoted by~$[W]$.
It is easy to see that $\Gr_d(V)$ consists of the single point~$[V]$
and~$\Gr_{1}(V)$ is nothing else than the projective space $\PP(V)$.

In the present paper we implement the above-discussed general method
of solving Problem~\ref{prob_main} for $G = \GL(V)$. In this case
there is a well-known description of the map from $\mathscr
F(\GL(V))$ to the set of nilpotent orbits in~$\mathfrak{gl}(V)$, as
well as the partial order on the latter set (see details
in~\S\,\ref{sect_PO_on_FV}). In particular, the following
proposition holds (see Corollary~\ref{crl_equiv_on_V-flags}):

\begin{proposition} \label{prop_permutations}
Let $\mathbf a$ and~$\mathbf b$ be two compositions of~$d$. The
following conditions are equivalent:

\textup{(a)} the varieties $\Fl_{\mathbf a}(V)$ and~$\Fl_{\mathbf
b}(V)$ are nil-equivalent;

\textup{(b)} $\mathbf a$ and~$\mathbf b$ can be obtained from each
other by a permutation \textup{(}in particular, $\mathbf a$
and~$\mathbf b$ contain the same number of elements\textup{)}.
\end{proposition}

In view of Theorem~\ref{thm_equiv_spherical},
Proposition~\ref{prop_permutations} implies the following theorem.

\begin{theorem} \label{thm_sphericity_equivalence}
Let $\mathbf a$ and $\mathbf b$ be two compositions of~$d$ obtained
from each other by a permutation and let $K \subset \GL(V)$ be an
arbitrary connected reductive subgroup. The following conditions are
equivalent:

\textup{(a)} the action $K : \Fl_{\mathbf a}(V)$ is spherical;

\textup{(b)} the action $K : \Fl_{\mathbf b}(V)$ is spherical.
\end{theorem}

Theorem~\ref{thm_sphericity_equivalence} reduces the problem of
describing all spherical actions on $V$-flag varieties to the case
of varieties $\Fl_{\mathbf a}(V)$ such that the composition $\mathbf
a = (a_1, \ldots, a_s)$ satisfies the inequalities $a_1 \le \ldots
\le a_s$.

It is easy to see that the $K$-sphericity of $\PP(V)$ is equivalent
to the sphericity of the $(K \times\nobreak \FF^\times)$-module~$V$,
where $\FF^\times$ acts on~$V$ by scalar transformations. Therefore
a description of all spherical actions on $\PP(V)$ is a trivial
consequence of the known classification of spherical modules that
was obtained in the papers~\cite{Kac},~\cite{BR}, and~\cite{Lea}.
(This classification plays a key role in our paper and is presented
in \S\,\ref{subsec_spherical_modules}.) As $\PP(V) = \Gr_1(V)$, to
complete the description of all spherical actions on $V$-flag
varieties it suffices to restrict ourselves to the case of varieties
$\Fl_{\mathbf a}(V)$ such that the composition $\mathbf a$ is
nontrivial and distinct from $(1, d - 1)$.

Before we state the main result of this paper, we need to introduce
one more notion and some additional notation.

Let $K_1, K_2$ be connected reductive groups, $U_1$ a $K_1$-module,
and $U_2$ a $K_2$-module. Consider the corresponding representations
$$
\rho_1 \colon K_1 \to \GL(U_1) \quad \text{ and } \quad \rho_2
\colon K_2 \to \GL(U_2).
$$
Following Knop (see~\cite[\S\,5]{Kn98}), we say that the pairs
$(K_1, U_1)$ are $(K_2, U_2)$ \textit{geometrically equivalent} if
there exists an isomorphism $U_1 \xrightarrow{\sim} U_2$ identifying
the groups $\rho_1(K_1) \subset \GL(U_1)$ and $\rho_2(K_2) \subset
\GL(U_2)$. In other words, the pairs $(K_1, U_1)$ and $(K_2, U_2)$
are geometrically equivalent if and only if they define the same
linear group. For example, every pair $(K, U)$ is geometrically
equivalent to the pair $(K, U^*)$ (where $U^*$ is the $K$-module
dual to~$U$), the pair $(\SL_2, \mathrm{S}^2 \FF^2)$ is
geometrically equivalent to the pair $(\SO_3, \FF^3)$, and the pair
$(\SL_2 \times \SL_2, \FF^2 \otimes \FF^2)$ is geometrically
equivalent to the pair $(\SO_4, \FF^4)$.

Let $K$ be a connected reductive subgroup of~$\GL(V)$ and let $K'$
be the derived subgroup of~$K$. We denote by $C$ the connected
component of the identity of the center of~$K$. Let $\mathfrak X(C)$
denote the character group of~$C$ (in additive notation). We
regard~$V$ as a $K$-module and fix a decomposition $V = V_1 \oplus
\ldots \oplus V_r$ into a direct sum of simple $K$-submodules. For
every $i = 1, \ldots, r$ we denote by $\chi_i$ the character of~$C$
via which~$C$ acts on~$V_i$.

The following theorem is the main result of this paper.

\begin{theorem} \label{thm_main_theorem}
Suppose that $\mathbf a = (a_1, \ldots, a_s)$ is a nontrivial
composition of~$d$ such that $a_1 \le \ldots \le a_s$ and $\mathbf a
\ne (1, d - 1)$. Then the variety $\Fl_{\mathbf a}(V)$ is
$K$-spherical if and only if the following conditions hold:

\textup{(1)} the pair $(K', V)$, considered up to a geometrical
equivalence, and the tuple $(a_1, \ldots, a_{s-1})$ are contained in
Table~\textup{\ref{table_main_result}};

\textup{(2)} the group $C$ satisfies the conditions listed in the
fourth column of Table~\textup{\ref{table_main_result}}.
\end{theorem}

\begin{table}[h!]

\begin{center}

\caption{} \label{table_main_result}

\begin{tabular}{|c|c|c|c|c|}
\hline

No. & $(K', V)$ & $(a_1, \ldots, a_{s-1})$ & Conditions on $C$ & Note \\

\hline

\multicolumn{5}{|c|}{$s = 2$ (Grassmannians)}\\

\hline

1 & $(\SL_n, \FF^n)$ & $(k)$ & & $n \ge 4$ \\

\hline

2 & $(\Sp_{2n}, \FF^{2n})$ & $(k)$ & & $n \ge 2$ \\

\hline

3 & $(\SO_n, \FF^{n})$ & $(k)$ & & $n \ge 4$ \\

\hline

4 & $(\Spin_7, \FF^8)$ & $(2)$ & & \\

\hline

5 & $(\Sp_{2n}, \FF^{2n} \oplus \FF^1)$ & $(k)$ &

& $n \ge 2$ \\

\hline

6 & $(\SL_n \times \SL_m, \FF^n \oplus \FF^m)$ & $(k)$ & $\chi_1 \ne
\chi_2$ for $n = m = k$ &
\begin{tabular}{c}
$n \ge m \ge 1$, \\ $n + m \ge 4$
\end{tabular}
\\

\hline

7 & $(\Sp_{2n} \times \SL_m, \FF^{2n} \oplus \FF^m)$ & $(2)$ &
$\chi_1 \ne \chi_2$ for $m = 2$ & $n,m \ge 2$\\

\hline

8 & $(\Sp_{2n} \times \SL_m, \FF^{2n} \oplus \FF^m)$ & $(3)$ &
$\chi_1 \ne \chi_2$ for $m \le 3$
& $n, m \ge 2$\\

\hline

9 & $(\Sp_4 \times \SL_m, \FF^4 \oplus \FF^m)$ & $(k)$ & $\chi_1 \ne
\chi_2$ for $k = m = 4$

& $m, k \ge 4$ \\

\hline

10 &
\begin{tabular}{c}
$(\Sp_{2n} \times \Sp_{2m}$,\\
$\FF^{2n} \oplus \FF^{2m})$
\end{tabular}
& $(2)$ & $\chi_1 \ne \chi_2$ & $n \ge m \ge 2$ \\

\hline

11 &
\begin{tabular}{c}
$(\SL_n \times \SL_m$, \\
$\FF^n \oplus \FF^m \oplus \FF^1)$
\end{tabular}
& $(k)$ &
\begin{tabular}{c}
$\chi_2 - \chi_1, \chi_3 - \chi_1$ \\ lin. ind.
for $k = n$\\
\hline $\chi_2 \ne \chi_3$ for $m \le k < n$
\end{tabular} &
\begin{tabular}{c}
$n \ge m \ge 1$, \\ $n \ge 2$
\end{tabular}
\\

\hline

12 &
\begin{tabular}{c}
$(\SL_n \times \SL_m \times \SL_l$,\\ $\FF^n \oplus \FF^m \oplus
\FF^l)$
\end{tabular}
& $(2)$ &

\begin{tabular}{c}
$\chi_2 - \chi_1, \chi_3 - \chi_1$ \\ lin. ind. for $n = 2$ \\

\hline

$\chi_2 \ne \chi_3$ for \\
$n \ge 3$, $m \le 2$ \\
\end{tabular}

&
\begin{tabular}{c}
$n {\ge} m {\ge} l {\ge} 1$, \\ $n \ge 2$
\end{tabular}
\\

\hline

13 &
\begin{tabular}{c}
$(\Sp_{2n} \times \SL_m \times \SL_l$,\\ $\FF^{2n} \oplus \FF^m
\oplus \FF^l)$
\end{tabular}
& $(2)$ &

\begin{tabular}{c}
$\chi_2 - \chi_1, \chi_3 - \chi_1$ \\ lin. ind. for $m \le 2$ \\

\hline

$\chi_1 \ne \chi_3$ for \\ $m \ge 3, l \le 2$ \\
\end{tabular}

&
\begin{tabular}{c}
$n \ge 2$, \\ $m \ge l \ge 1$
\end{tabular}
\\

\hline

14 &
\begin{tabular}{c}
$(\Sp_{2n} \times \Sp_{2m} \times \SL_l$,\\ $\FF^{2n} \oplus
\FF^{2m} \oplus \FF^l)$
\end{tabular}
& $(2)$ &

\begin{tabular}{c}
$\chi_2 - \chi_1, \chi_3 - \chi_1$ \\ lin. ind. for $l \le 2$ \\

\hline

$\chi_1 \ne \chi_2$ for $l \ge 3$ \\
\end{tabular}

&
\begin{tabular}{c}
$n \ge m \ge 2$, \\ $l \ge 1$
\end{tabular}
\\

\hline

15 &
\begin{tabular}{c}
$(\Sp_{2n} \times \Sp_{2m} \times \Sp_{2l}$, \\ $\FF^{2n} \oplus
\FF^{2m} \oplus \FF^{2l})$
\end{tabular}
& $(2)$ &
\begin{tabular}{c}
$\chi_2 - \chi_1, \chi_3 - \chi_1$ \\ lin. ind.
\end{tabular}
& $n \,{\ge} m \,{\ge} l\, {\ge} 2$
\\

\hline

\multicolumn{5}{|c|}{$s \ge 3$} \\

\hline

16 & $(\SL_n, \FF^n)$ & $(a_1, \ldots, a_{s-1})$ & & $n \ge 3$ \\

\hline

17 & $(\Sp_{2n}, \FF^{2n})$ & $(1, a_2)$ & & $n \ge 2$ \\

\hline

18 & $(\Sp_{2n}, \FF^{2n})$ & $(1,1,1)$ & & $n \ge 2$ \\

\hline

19 & $(\SL_n, \FF^n \oplus \FF^1)$ & $(a_1, \ldots, a_{s-1})$ &
$\chi_1 \ne \chi_2$ for $s = n+1$ & $n \ge 2$ \\

\hline

20 & $(\SL_n \times \SL_m, \FF^n \oplus \FF^m)$ & $(1, a_2)$ &
$\chi_1 \ne \chi_2$ for $n = 1 + a_2$ & $n \ge m \ge 2$ \\

\hline

21 & $(\SL_n \times \SL_2, \FF^n \oplus \FF^2)$ & $(a_1, a_2)$ &
\begin{tabular}{c}
$\chi_1 \ne \chi_2$ for \\
$n = 4$, $a_1 = a_2 = 2$
\end{tabular}
&
\begin{tabular}{c}
$n \ge 4$,\\ $a_1 \ge 2$
\end{tabular}
\\

\hline

22 & \begin{tabular}{c}
$(\Sp_{2n} \times \SL_m$, \\
$\FF^{2n} \oplus \FF^m)$
\end{tabular}
& $(1,1)$ & $\chi_1 \ne \chi_2$ for $m \le 2$ &
\begin{tabular}{c}
$n \ge 2$, \\
$m \ge 1$
\end{tabular}
\\

\hline

23 & \begin{tabular}{c}
$(\Sp_{2n} \times \Sp_{2m}$, \\
$\FF^{2n} \oplus \FF^{2m})$
\end{tabular}
&
$(1,1)$ & $\chi_1 \ne \chi_2$ & $n \ge m \ge 2$\\

\hline

\end{tabular}

\end{center}

\end{table}

Theorem~\ref{thm_main_theorem} is a union of
Theorems~\ref{thm_gr2_case_of_simple_V},
\ref{thm_gr2_case_of_nonsimple_V}, \ref{thm_grassmannians},
and~\ref{thm_non-grassmannians}.

Let us explain the notation and conventions used in
Table~\ref{table_main_result}. In each case we assume that the
$i$-th factor of $K'$ acts on the $i$-th direct summand of~$V$.
Further, we assume that the group of type $\SL_q$ (resp. $\Sp_{2q}$,
$\SO_q$) naturally acts on $\FF^q$ (resp. $\FF^{2q}$,~$\FF^q$). In
row~4 the group $\Spin_7$ acts on~$\FF^8$ via the spin
representation. If the description of the tuple $(a_1, \ldots,
a_{s-1})$ given in the third column contains parameters, then these
parameters may take any admissible values (that is, any values such
that $a_1 \le \ldots \le a_s$ and $\mathbf a \ne (1, d-1)$). In
particular, in rows~16 and~19 any composition~$\mathbf a$ satisfying
the above restrictions is possible. The empty cells in the fourth
column mean that there are no conditions on~$C$, that is, the
characters $\chi_1, \ldots, \chi_r$ may be arbitrary. The
abbreviation ``lin. ind.'' stands for ``are linearly independent
in~$\mathfrak X(C)$''.

Our proof of Theorem~\ref{thm_main_theorem} is based on an analysis
of the partial order $\preccurlyeq$ on the set~$\mathscr F(\GL(V))
/\! \sim$. Here the starting points are the following facts:

(1) if $X \in \mathscr F(\GL(V))$ then $\bl X \br \succcurlyeq \bl
\PP(V) \br$;

(2) if $X \in \mathscr F(\GL(V))$, $\bl X \br \ne \bl \PP(V) \br$,
and $d \ge 4$, then $\bl X \br \succcurlyeq \bl \Gr_2(V) \br$.

In view of Theorem~\ref{thm_prec_spherical}, facts~(1) and~(2) imply
the following results, which hold for any nontrivial $V$-flag
variety $X$ and any connected reductive subgroup $K \subset \GL(V)$:

($1'$) if $X$ is a $K$-spherical variety then so is $\PP(V)$ (see
Proposition~\ref{prop_P(V)_is_spherical});

($2'$) if $X$ is a $K$-spherical variety, $\bl X \br \ne \bl \PP(V)
\br$, and $d \ge 4$, then $\Gr_2(V)$ is a $K$-spherical variety (see
Proposition~\ref{prop_Gr2_is_spherical}).

Assertion~($1'$) means that a necessary condition for $K$-sphericity
of a nontrivial $V$-flag variety is that $V$ be a spherical $(K
\times \FF^\times)$-module, where $\FF^\times$ acts on $V$ by scalar
transformations (see Corollary~\ref{crl_V_is_spherical}).
Assertion~($2'$) implies that the first step in the proof of
Theorem~\ref{thm_main_theorem} is a description of all spherical
actions on $\Gr_2(V)$.

We note that assertion~($1'$) was in fact proved
in~\cite[Theorem~5.8]{Pet} using the ideas presented in this paper.

The list of subgroups of~$\GL(V)$ acting spherically on $\Gr_2(V)$
(see Theorems~\ref{thm_gr2_case_of_simple_V}
and~\ref{thm_gr2_case_of_nonsimple_V}) turns out to be substantially
shorter than that of subgroups acting spherically on~$\PP(V)$. This
makes the subsequent considerations easier and enables us to
complete the description of all spherical actions on nontrivial
$V$-flag varieties that are nil-equivalent to neither $\PP(V)$
nor~$\Gr_2(V)$ (see Theorems~\ref{thm_grassmannians}
and~\ref{thm_non-grassmannians}).

The present paper is organized as follows. In~\S\,\ref{sect_GFV_NO}
we recall some facts on Poisson and symplectic varieties and then,
using them, we prove the $K$-sphericity criterion of a generalized
flag variety~$X$ in terms of the action $K : \mathcal N(X)$. (This
criterion implies Theorem~\ref{thm_equiv_spherical}.) At~the end of
\S\,\ref{sect_GFV_NO} we prove Theorem~\ref{thm_prec_spherical}.
In~\S\,\ref{sect_PO_on_FV} we study the nil-equivalence relation and
the partial order on nil-equivalence classes in the case $G =
\GL(V)$. We also discuss a transparent interpretation of this
partial order in terms of Young diagrams. In~\S\,\ref{sect_tools} we
collect all auxiliary results that will be needed in our proof of
Theorem~\ref{thm_main_theorem}.
In~\S\,\ref{sect_known_classifications} we present two known
classifications that will be used in the proof of
Theorem~\ref{thm_main_theorem}: the first one is the classification
of spherical modules and the second one is the classification of
Levi subgroups in~$\GL(V)$ acting spherically on $V$-flag varieties.
We prove Theorem~\ref{thm_main_theorem}
in~\S\,\ref{sect_proof_of_main_theorem}. At last,
Appendix~\ref{sect_appendix} contains the most complicated technical
proofs of some statements from~\S\,\ref{sect_tools}.

The authors express their gratitude to E.\,B.~Vinberg and
I.\,B.~Penkov for useful discussions, as well as to the referee for
a careful reading of the previous version of this paper and valuable
comments.

\subsection*{Basic notation and conventions}

In this paper all varieties, groups, and subgroups are assumed to be
algebraic. All topological terms relate to the Zariski topology. The
Lie algebras of groups denoted by capital Latin letters are denoted
by the corresponding small Gothic letters. All algebras (except for
Lie algebras) are assumed to be associative, commutative, and with
identity. All vector spaces are assumed to be finite-dimensional.

Let $V$ be a vector space. Any nondegenerate skew-symmetric bilinear
form on~$V$ will be called a \textit{symplectic form}. If $\Omega$
is a fixed symplectic form on~$V$, then for every subspace $W
\subset V$ we shall denote by $W^\perp$ the skew-orthogonal
complement to~$W$ with respect to~$\Omega$.

All the notation and conventions used in
Table~\ref{table_main_result} will be also used in all other tables
appearing in this paper.

Notation:

$|X|$ is the cardinality of a finite set~$X$;

$V^*$ is the space of linear functions on a vector space~$V$

$\mathrm S^2 V$ is the symmetric square of a vector space~$V$;

$\wedge^2 V$ is the exterior square of a vector space~$V$;

$\langle v_1, \ldots, v_k \rangle$ is the linear span of vectors
$v_1, \ldots, v_k$ of a vector space~$V$;

$G : X$ denotes an action of a group $G$ on a variety~$X$;

$G_x$ is the stabilizer of a point $x \in X$ under an action $G :
X$;

$G'$ is the derived subgroup of a group~$G$;

$G^0$ is the connected component of the identity of a group~$G$;

$\mathfrak X(G)$ is the character group of a group~$G$ (in additive
notation);

$\mathrm{S}(\mathrm{L}_n \times \mathrm{L}_m)$ is the subgroup
in~$\GL_{n+m}$ equal to $(\GL_n \times \GL_m) \cap \SL_{n+m}$;

$\mathrm{S}(\mathrm{O}_n \times \mathrm{O}_m)$ is the subgroup in
$\mathrm{O}_{n+m}$ equal to $(\mathrm{O}_n \times \mathrm{O}_m) \cap
\SO_{n+m}$;

$\overline Y$ is the closure of a subset~$Y$ of a variety~$X$;

$\rk G$ is the rank of a reductive group~$G$, that is, the dimension
of a maximal torus of~$G$;

$X^\reg$ is the set of regular points of a variety~$X$;

$\FF[X]$ is the algebra of regular functions on a variety~$X$;

$\FF(X)$ is the field of rational functions on a variety~$X$;

$A^G$ is the algebra of invariants of an action of a group $G$ on an
algebra~$A$;

$\Quot A$ is the field of fractions of an algebra $A$ without zero
divisors;

$\Spec A$ is the spectrum of a finitely generated algebra $A$
without nilpotents, that is, the affine algebraic variety whose
algebra of regular functions is isomorphic to~$A$;

$T_x X$ is the tangent space of a variety $X$ at a point~$x$;

$T^*_x X = (T_x X)^*$ is the cotangent space of a variety $X$ at a
point~$x$;

$T^*X$ is the cotangent bundle of a smooth variety~$X$;

$P^\top$ is the transpose matrix of a matrix~$P$.

\section{Generalized flag varieties and nilpotent orbits}
\label{sect_GFV_NO}

Throughout this section we fix an arbitrary connected reductive
group $G$ and an arbitrary connected reductive subgroup $K
\subset\nobreak G$.

\subsection{Poisson and symplectic varieties}

In this subsection we gather all the required information on Poisson
and symplectic varieties. The information here is taken from
\cite[\S\S\,II.1--II.3]{Vin01}.

An algebra $A$ is said to be a \textit{Poisson algebra} if $A$ is
equipped with a bilinear anticommutative operation $\{ \,\cdot \,
{,}\,\cdot\,\} \colon A\times A \to A$ (called a \textit{Poisson
bracket}) satisfying the identities
$$
\text{\renewcommand{\arraystretch}{1.2}
\begin{tabular}{cc}
$\lbrace f, gh \rbrace = \lbrace f, g \rbrace h + \lbrace f, h
\rbrace g,$ & \qquad(Leibniz identity)\\
$\lbrace f, \lbrace g, h \rbrace \rbrace + \lbrace g, \lbrace h, f
\rbrace \rbrace + \lbrace h, \lbrace f, g \rbrace \rbrace = 0$ &
\qquad(Jacobi identity)
\end{tabular}
}
$$
for any elements $f, g, h \in A$.

An irreducible variety $M$ is said to be a \textit{Poisson variety}
if the structure sheaf of $M$ is endowed with a structure of a sheaf
of Poisson algebras. We note that the Poisson structure on $M$ is
uniquely determined by the Poisson bracket induced in the
field~$\mathbb F(M)$.

A smooth irreducible variety $M$ together with a nondegenerate
closed $2$-form $\omega$ is said to be a \textit{symplectic
variety}. In this situation, the $2$-form $\omega$ is said to be the
\textit{structure $2$-form}.

Let $M$ be a Poisson variety. There is a unique bivector field
$\mathcal B$ on $M^\reg$ with the following property:
\begin{equation} \label{eqn_Poisson_bracket_via_bivector}
\lbrace f, g \rbrace = \mathcal B(df, dg)
\end{equation}
for any $f,g \in \FF(M)$. The bivector $\mathcal B$ is said to be
the \textit{Poisson bivector}. If $\mathcal B$ is nondegenerate at
each point of a nonempty open subset~$Z \subset M^\reg$, then the
$2$-form $\omega = (\mathcal B^\top)^{-1}$ (the formula relates the
matrices of $\omega$ and $\mathscr B$ in dual bases of the spaces
$T^*_pM$ and $T_pM$ for each point $p\in M$) defines a symplectic
structure on~$Z$.

Conversely, if $M$ is a symplectic variety with structure
$2$-form~$\omega$ then the bivector $\mathcal B =
(\omega^\top)^{-1}$ defines a Poisson structure on $M$ by the
formula~(\ref{eqn_Poisson_bracket_via_bivector}), so that every
symplectic variety is Poisson.

A morphism $\varphi \colon M \to M'$ of Poisson varieties is said to
be \textit{Poisson} if for every open subset $U\subset M'$ the
corresponding homomorphism $\varphi^* \colon \mathbb F[U] \to
\mathbb F[\varphi^{-1}(U)]$ is a homomorphism of Poisson algebras.
If $\varphi$ is surjective, then $\varphi$ is Poisson if and only if
the pushforward of the Poisson bivector on $M$ coincides with the
Poisson bivector on~$M'$. In particular, if $\varphi \colon M \to
M'$ is a Poisson morphism of symplectic varieties that is a
covering, then the pullback of the structure $2$-form on~$M'$
coincides with the structure $2$-form on~$M$.

The space $\mathfrak g^*$ is endowed with a natural Poisson
structure. In view of the Leibnitz identity, this structure is
uniquely determined by specifying the values of the Poisson bracket
on linear functions, that is, on the space $(\mathfrak
g^*)^*\simeq\mathfrak g$. For $\xi, \eta\in\mathfrak g$ one defines
$\{ \xi, \eta \}=[\xi, \eta]$.

It is well known that for an arbitrary $G$-orbit $O$ in~$\mathfrak
g^*$ the restriction of the Poisson bivector to~$O$ is well defined
and nondegenerate, which induces a symplectic structure on~$O$. This
structure is determined by the $2$-form~$\omega_O$ given at a point
$\alpha \in \mathfrak g^*$ by the formula
$$
\omega_O(\mathrm{ad}^*(\xi) \alpha, \mathrm{ad}^*(\eta)
\alpha)=\alpha([\xi, \eta]),
$$
where $\xi, \eta\in\mathfrak g$ and $\mathrm{ad}^*$ is the coadjoint
representation of the algebra~$\mathfrak g$. The form $\omega_O$ is
said to be the \textit{Kostant--Kirillov form}.

Let $X$ be a smooth irreducible variety. Then its cotangent bundle
$$
T^* X = \lbrace (x, p) \mid x \in X, p \in T^*_x X \rbrace
$$
is also a smooth irreducible variety. There is a canonical
symplectic structure on $T^* X$. The structure $2$-form~$\omega_X$
can be expressed in the form $\omega_X = - d\theta$, where $\theta$
is the $1$-form defined as follows. Let $\pi \colon T^* X \to X$ be
the canonical projection and let $d\pi$ be its differential. Let
$\xi$ be a tangent vector to $T^* X$ at a point $(x, p)$. Then
$\theta(\xi) = p(d\pi(\xi))$.

\subsection{Some properties of symplectic
$G$-varieties} \label{subsec_properties_SV}

Let $X$ be a smooth irreducible $G$-variety. The action $G : X$
naturally induces an action $G :\nobreak T^* X$ preserving the
structure $2$-form~$\omega_X$.

Let $\mathscr V(X)$ be the Lie algebra of vector fields on~$X$. The
action of~$G$ on~$X$ defines the Lie algebra homomorphism $\tau_X
\colon \mathfrak g \to \mathscr V(X)$ taking each element $\xi \in
\mathfrak g$ to the corresponding velocity field on~$X$. For all
$\xi \in \mathfrak g$ and $x \in X$ let $\xi x$ denote the value of
the field $\tau_X(\xi)$ at~$x$.

The map $\Phi \colon T^* X \to \mathfrak g^*$ given by the formula
$$
(x, p) \mapsto [\xi \mapsto p(\xi x)], \quad \quad \text{where } x
\in X, p \in T^*_x X, \xi \in \mathfrak g,
$$
is said to be the \textit{moment map}.

\begin{proposition}[{see~\cite[\S\,II.2.3, Proposition~2]{Vin01}}]
\label{prop_moment_map_Poisson} The map~$\Phi$ is a $G$-equivariant
Poisson morphism.
\end{proposition}

Let $V$ be a vector space with a given symplectic form~$\Omega$. A
subspace $W \subset V$ is said to be \textit{isotropic} if the
restriction of~$\Omega$ to~$W$ is identically zero and
\textit{coisotropic} if the skew-orthogonal complement of $W$ is
isotropic.

Let $M$ be a symplectic variety with structure $2$-form~$\omega$. A
smooth irreducible locally closed subvariety $Z \subset M$ is said
to be \textit{coisotropic} if the subspace $T_z Z$ is coisotropic in
$T_z M$ for each point $z \in Z$.

\begin{definition} \label{dfn_coisotropic}
An action $G : X$ preserving the structure $2$-form~$\omega$ is said
to be \textit{coisotropic} if orbits of general position for this
action are coisotropic.
\end{definition}

The following theorem is implied by~\cite[Theorem~7.1]{Kno90}, see
also~\cite[\S\,II.3.4, Theorem~2, Corollary~1]{Vin01}.

\begin{theorem} \label{thm_spherical&coisotropic}
Let $X$ be a smooth irreducible $G$-variety. The following
conditions are equivalent:

\textup{(a)} the action $G : X$ is spherical;

\textup{(b)} the action $G : T^* X$ is coisotropic.
\end{theorem}

A subset $A'$ of a Poisson algebra $A$ is said to be
\textit{Poisson-commutative} if $\{ f, g\}=0$ for all $f,g\in A'$.

\begin{proposition}[{see~\cite[\S\,II.3.2, Proposition~5]{Vin01}}]
\label{prop_coisotropic_Poisson_commutative} Let $M$ be a symplectic
$G$-variety such that the structure $2$-form is $G$-invariant. The
following conditions are equivalent:

\textup{(a)} the action $G : M$ is coisotropic;

\textup{(b)} the field $\FF(M)^G$ is Poisson-commutative.
\end{proposition}

\subsection{The $K$-sphericity criterion for a generalized
flag variety} \label{subsec_sphericity_criterion}

The main result of this subsection is
Theorem~\ref{thm_spherical&coisotropic_orbit}.

We shall identify $\mathfrak g$ and $\mathfrak g^*$ via a bilinear
form given by the trace of the product of two linear operators in a
fixed faithful linear representation of~$\mathfrak g$.

Let $P \subset G$ be a parabolic subgroup and let $N$ be the
unipotent radical of~$P$. Recall that in Introduction we defined the
nilpotent orbit $\mathcal N(G/P) \subset \mathfrak g$,
see~(\ref{eqn_nilp_orbit}). Put $o = eP \in G / P$. Let $\Phi_P
\colon T^* (G / P) \to \mathfrak g$ be the moment map corresponding
to the natural action $G : G / P$.

\begin{proposition} \label{prop_image_of_moment_map}
The image of $\Phi_P$ coincides with the closure of the orbit
$\mathcal N(G/P) \subset \mathfrak g$. Moreover, $\Phi_P$ is finite
over $\mathcal N (G / P)$.
\end{proposition}

\begin{proof}
In view of the identifications $T^*_o (G / P) \simeq (\mathfrak g /
\mathfrak p)^* \simeq \mathfrak n$, it follows from
\cite[\S\,II.2.3, Example~5]{Vin01} that $\Phi_P(T^*_o(G / P)) =
\mathfrak n$. This together with the $G$-equivariance of the
map~$\Phi_P$ (see Proposition~\ref{prop_moment_map_Poisson}) implies
that the image of~$\Phi_P$ coincides with the subset $G \mathfrak n
\subset\nobreak \mathfrak g$ and, in particular, is irreducible. By
\cite[Proposition~6(b)]{Rich74} the set $G\mathfrak n$ is closed
in~$\mathfrak g$ and has dimension~$2\dim\mathfrak n$. By its
definition, the orbit $\mathcal N(G/P)$ is dense in~$G\mathfrak n$,
therefore $\dim\mathcal N(G/P)=\dim T^* (G/P)$. The latter means
that the map $\Phi_P$ is finite over~$\mathcal N(G/P)$.
\end{proof}

\begin{theorem} \label{thm_spherical&coisotropic_orbit}
The following conditions are equivalent:

\textup{(a)} the action $K : G / P$ is spherical;

\textup{(b)} the action $K : \mathcal N(G / P)$ is coisotropic.
\end{theorem}

\begin{proof}
By Proposition~\ref{prop_image_of_moment_map} the subset
$U=\Phi_P^{-1}(\mathscr N(G/ P))$ is open in~$T^* (G/P)$ and the map
$\left. \Phi_P \right|_U \colon U \to\mathcal N(G/P)$ is a covering.
It follows from Proposition~\ref{prop_moment_map_Poisson} that
$\left. \Phi_P \right|_U$ is a \hbox{$G$-equivariant} Poisson
morphism of symplectic varieties, therefore the image of the
structure $2$-form on $\mathcal N(G/P)$ coincides with the structure
$2$-form on~$U$. Consequently, the action $K :\mathcal N(G/P)$ is
coisotropic if and only if the action $K : U$ is so or,
equivalently, the action $K : T^* (G/P)$ is so. In view of
Theorem~\ref{thm_spherical&coisotropic} the latter holds if and only
if the action $K : G / P$ is spherical.
\end{proof}

\subsection{Proof of Theorem~\ref{thm_prec_spherical}}
\label{subsec_proof_of_thm_prec}

By Theorem~\ref{thm_spherical&coisotropic_orbit}, the proof of
Theorem~\ref{thm_prec_spherical} reduces to that of the following
proposition.

\begin{proposition} \label{prop_two_nilpotent_orbits_coisotropic}
Let $O_1, O_2$ be two nilpotent $G$-orbits in~$\mathfrak g$ such
that $O_1 \subset \overline{O}_2$. If the action $K : O_2$ is
coisotropic then so is the action $K : O_1$.
\end{proposition}

In order to prove this proposition, we shall need
Lemma~\ref{lemma_integral_closure} and Proposition~\ref{prop_Losev}
given below.

\begin{lemma}[{see~\cite[Proposition~1.2]{BK}}]
\label{lemma_integral_closure}
For every $G$-orbit $O \subset \mathfrak g$ the algebra $\FF[O]$ is
the integral closure of the algebra $\FF[\overline{O}]$ in the field
$\FF(O)$. In particular, $\FF[O]$ is integrally closed.
\end{lemma}

\begin{proposition} \label{prop_Losev}
For every $G$-orbit $O \subset \mathfrak g$, one has $\FF(O)^K =
\Quot (\FF[O]^K)$.
\end{proposition}

\begin{proof}
Lemma~\ref{lemma_integral_closure} implies that the variety $\Spec
\FF[O]$ is normal. Now the required result follows
from~\cite[Corollary~3.4.1]{Los} and~\cite[Theorem~1.2.4,
part~1]{Los}.
\end{proof}

\begin{proof}[Proof of Proposition~\textup{\ref{prop_two_nilpotent_orbits_coisotropic}}]
Suppose that the action $K : O_2$ is coisotropic. Then by
Proposition~\ref{prop_coisotropic_Poisson_commutative} the field
$\FF(O_2)^K$ is Poisson-commutative. Therefore the algebra
$\FF[\overline{O}_2]^K$ is Poisson-commutative as well. Consider the
restriction map $\FF[\overline{O}_2] \to \FF[\overline{O}_1]$. It is
surjective and a $K$-module homomorphism, hence the image of
$\FF[\overline{O}_2]^K$ coincides with $\FF[\overline{O}_1]^K$. It
follows that the latter algebra is Poisson-commutative. Let us show
that the algebra $\FF[O_1]^K$ is also Poisson-commutative. Let $f
\in \FF[O_1]^K$ be an arbitrary element. It follows from
Lemma~\ref{lemma_integral_closure} that $f$ satisfies an equation of
the form
\begin{equation} \label{eqn_integral_dependence}
f^n + c_{n-1}f^{n-1} + \ldots + c_1 f + c_0 = 0,
\end{equation}
where $n \ge 1$ and $c_0, \ldots, c_{n-1} \in \FF[\overline{O}_1]$.
Applying the operator of ``averaging over~$K$'' (which is also known
as the Reynolds operator; see~\cite[\S\,3.4]{VP}) we may assume that
$c_0, \ldots, c_{n-1} \in \FF[\overline{O}_1]^K$. Moreover, we shall
assume that the number~$n$ is minimal among all equations of the
form~(\ref{eqn_integral_dependence}). For an arbitrary element $g
\in \FF[\overline{O}_1]^K$ we have $\lbrace c_i, g \rbrace = 0$ for
all $i = 0, \ldots, n-1$. Therefore, applying the Poisson bracket
with $g$ to both sides of~(\ref{eqn_integral_dependence}), we obtain
$$
(nf^{n-1} + (n-1)c_{n-1}f^{n-2} + \ldots + c_1) \lbrace f, g \rbrace
= 0.
$$
Since $n$ is minimal, the expression $nf^{n-1} + (n-1)c_{n-1}f^{n-2}
+ \ldots + c_1$ is different from zero, hence $\lbrace f, g \rbrace
= 0$. Consequently, $\lbrace \FF[O_1]^K, \FF[\overline{O}_1]^K
\rbrace = 0$. Applying the same argument to an arbitrary element $g
\in \FF[O_1]^K$, again we get $\lbrace f, g \rbrace =\nobreak 0$,
hence the algebra $\FF[O_1]^K$ is Poisson-commutative. Then
Proposition~\ref{prop_Losev} implies that the field $\FF(O_1)^K$ is
also Poisson-commutative. Applying
Proposition~\ref{prop_coisotropic_Poisson_commutative} we find that
the action $K : O_1$ is coisotropic.
\end{proof}

\begin{remark}
In the case $G = \GL_n$ (or $\SL_n$) the proof of
Proposition~\ref{prop_two_nilpotent_orbits_coisotropic} simplifies.
Namely, as was proved in~\cite{KP}, in this case the closure of
every $G$-orbit in~$\mathfrak g$ is normal. Hence by
Lemma~\ref{lemma_integral_closure} we have $\FF[O_1] =
\FF[\overline{O}_1]$, and so $\FF[O_1]^K = \FF[\overline{O}_1]^K$.
\end{remark}

\section{The partial order on the set of nil-equivalence
classes\\of $V$-flag varieties} \label{sect_PO_on_FV}

Throughout this section we fix a vector space $V$ of dimension~$d$.

\subsection{Nilpotent orbits in~$\mathfrak{gl}(V)$}
\label{subsec_nilpotent_orbits}

A composition $(a_1, \ldots, a_s)$ of $d$ is said to be a
\textit{partition} if $a_1 \ge \ldots \ge\nobreak a_s$.

The following fact is well known.

\begin{theorem}
There is a bijection between the nilpotent orbits in
$\mathfrak{gl}(V)$ and the partitions of~$d$. Under this bijection,
the orbit corresponding to a partition $(a_1, \ldots, a_s)$ consists
of all matrices whose Jordan normal form has zeros on the diagonal
and the block sizes are $a_1, \ldots, a_s$ up to a permutation.
\end{theorem}

For every partition $\mathbf a = (a_1, \ldots, a_s)$ of~$d$ we
denote by $O_{\mathbf{a}}$ the corresponding nilpotent orbit
in~$\mathfrak{gl}(V)$.

We now introduce a partial order on the set of partitions of~$d$ in
the following way. For two partitions $\mathbf a = (a_1, \ldots,
a_s)$ and $\mathbf b = (b_1, \ldots, b_t)$ we write $\mathbf a
\preccurlyeq \mathbf b$ (or $\mathbf b \succcurlyeq \mathbf a$) if
$$
a_1 + \ldots + a_i \le b_1 + \ldots + b_i \text{ for all } i = 1,
\ldots, d.
$$
(In this formula we put $a_j = 0$ for $j > s$ and $b_j = 0$ for $j
> t$.)

\begin{theorem}[{see~\cite[Theorem~6.2.5]{CM}}]
\label{thm_order_orbits} Let $\mathbf a$ and $\mathbf b$ be two
partitions of~$d$. The following conditions are equivalent:

\textup{(a)} $O_{\mathbf a} \subset \overline{O}_{\mathbf b}$;

\textup{(b)} $\mathbf a \preccurlyeq \mathbf b$.
\end{theorem}

\subsection{The correspondence between $V$-flag varieties and
nilpotent orbits in~$\mathfrak{gl}(V)$}
\label{subsec_flags&nilpotent_orbits}

For each composition $\mathbf a = (a_1, \ldots, a_s)$ of~$d$ one
defines the \textit{dual} partition $\mathbf{a}^\top = (\widehat
a_1, \ldots, \widehat a_{t})$ of~$d$ by the following rule:
$$
\widehat a_i = |\lbrace j \mid a_j \ge i \rbrace|, \quad i=1,
\ldots, t.
$$
Obviously, for every composition $\mathbf{b}$ of~$d$ obtained
from~$\mathbf{a}$ by a permutation one has ${\mathbf{a}^\top =
\mathbf{b}^\top}$. Besides, it is not hard to see that the operation
$\mathbf{a} \mapsto \mathbf{\widehat a}$ is an involution on the set
of partitions.

\begin{proposition}[{see \cite[Lemma~6.3.1]{CM}}]
\label{prop_involution_order} Let $\mathbf a$ and $\mathbf b$ be two
partitions of~$d$. The following conditions are equivalent:

\textup{(a)} $\mathbf a \preccurlyeq \mathbf b$;

\textup{(b)} $\mathbf a^\top \succcurlyeq \mathbf b^\top$.
\end{proposition}

Let $\mathbf a=(a_1,\dots, a_s)$ be an arbitrary composition of~$d$.
We denote by~$\mathbf a^\natural$ the partition of~$d$ obtained from
$\mathbf a$ by arranging its elements in the non-increasing order.

We fix a basis $e_1, \dots, e_d$ of~$V$. Let $P_{\mathbf a}$ be the
parabolic subgroup of~$\GL(V)$ that is the stabilizer of the point
$(V_1, \dots, V_s) \in \Fl_{\mathbf a}(V)$, where $V_i = \langle
e_1, \dots, e_{a_1+\ldots+a_i} \rangle$ for $i = 1, \dots, s$. It is
easy to see that, relatively to the basis $e_1, \dots, e_d$,
$P_{\mathbf a}$ consists of all nondegenerate block upper-triangular
matrices whose diagonal blocks have sizes $a_1, \dots, a_s$.

\begin{proposition} \label{prop_nilp_comp}
One has $\mathcal N(\Fl_{\mathbf{a}}(V)) = O_{\mathbf a^\top}$.
\end{proposition}

\begin{proof}
Put $\mathbf c = \mathbf a^\natural$. It is easy to see that Levi
subgroups of $P_{\mathbf{a}}$ and~$P_{\mathbf c}$ are conjugate
in~$\GL(V)$. Then from~\cite[Theorem~2.7]{JR} (see
also~\cite[Theorem~7.1.3]{CM}) it follows that $\mathcal
N(\Fl_{\mathbf{a}}(V)) = \mathcal N(\Fl_{\mathbf c}(V))$. Further,
by~\cite[\S\,2.2, Theorem]{Kra} (see also~\cite[Theorem~7.2.3]{CM})
one has $\mathcal N(\Fl_{\mathbf c}(V)) = O_{\mathbf c^\top}$. Since
$\mathbf a^\top = \mathbf c^\top$, we obtain $\mathcal
N(\Fl_{\mathbf{a}}(V)) = O_{\mathbf a^\top}$.
\end{proof}

\begin{corollary} \label{crl_equiv_on_V-flags}
Let $\mathbf a$ and $\mathbf b$ be two compositions of~$d$. The
following conditions are equivalent:

\textup{(a)} $\Fl_{\mathbf a}(V) \sim \Fl_{\mathbf b}(V)$;

\textup{(b)} $\mathbf a^\natural = \mathbf b^\natural$.
\end{corollary}

\begin{proof}[Proof \upshape follows from
Proposition~\ref{prop_nilp_comp} and the fact that $\mathbf a^\top =
\mathbf b^\top$ if and only if $\mathbf a^\natural = \mathbf
b^\natural$]
\end{proof}

\begin{corollary} \label{crl_partial_order}
Let $\mathbf a$ and $\mathbf b$ be two compositions of~$d$. The
following conditions are equivalent:

\textup{(a)} $\bl \Fl_{\mathbf a}(V) \br \preccurlyeq \bl
\Fl_{\mathbf b}(V) \br$;

\textup{(b)} $\mathbf a^\top \preccurlyeq \mathbf b^\top$;

\textup{(c)} $\mathbf a^\natural \succcurlyeq \mathbf b^\natural$.
\end{corollary}

\begin{proof}
Equivalence of (a) and (b) follows from
Theorem~\ref{thm_order_orbits} and Proposition~\ref{prop_nilp_comp}.
Equivalence of (b) and (c) follows from
Proposition~\ref{prop_involution_order}.
\end{proof}

\subsection{Young diagrams of $V$-flag varieties}
\label{subsec_Young_diagrams}

With every partition $\mathbf a = (a_1, \ldots, a_s)$ we associate
the left-aligned Young diagram $\YD(\mathbf a)$ whose $i$-th row
from the bottom contains $a_i$ boxes. With every variety
$\Fl_{\mathbf a}(V)$ we associate the Young diagram $\YD(\mathbf
a^\natural)$. As an example, we show the Young diagrams of some
$\FF^6$-flag varieties in Figure~\ref{fig_diagrams}.

\begin{figure}[h]

\begin{center}

\begin{picture}(330,47)
\put(0,17){%
\multiput(0,0)(0,10){2}{\line(1,0){50}}%
\multiput(0,0)(10,0){6}{\line(0,1){10}}%
\put(0,20){\line(1,0){10}}%
\multiput(0,10)(10,0){2}{\line(0,1){10}}%
\put(12,-17){$\PP(\FF^6)$}%
}%
\put(80,17){%
\multiput(0,0)(0,10){2}{\line(1,0){40}}%
\multiput(0,0)(10,0){5}{\line(0,1){10}}%
\put(0,20){\line(1,0){20}}%
\multiput(0,10)(10,0){3}{\line(0,1){10}}%
\put(1,-17){$\Gr_2(\FF^6)$}%
}%
\put(151,17){%
\multiput(0,0)(0,10){3}{\line(1,0){30}}%
\multiput(0,0)(10,0){4}{\line(0,1){20}}%
\put(-3,-17){$\Gr_3(\FF^6)$}%
}%
\put(214,17){%
\multiput(0,0)(0,10){2}{\line(1,0){40}}%
\multiput(0,0)(10,0){5}{\line(0,1){10}}%
\multiput(0,20)(0,10){2}{\line(1,0){10}}%
\multiput(0,10)(10,0){2}{\line(0,1){20}}%
\put(-5,-17){$\Fl(1, 1; \FF^6)$}%
}%
\put(293,17){%
\multiput(0,0)(0,10){2}{\line(1,0){30}}%
\multiput(0,0)(10,0){4}{\line(0,1){10}}%
\put(0,20){\line(1,0){20}}%
\put(0,30){\line(1,0){10}}%
\multiput(0,10)(10,0){2}{\line(0,1){20}}%
\put(20,10){\line(0,1){10}}%
\put(-10,-17){$\Fl(1, 2; \FF^6)$}%
}%
\end{picture}

\caption{} \label{fig_diagrams}

\end{center}
\end{figure}

The partial order on the set of nil-equivalence classes of $V$-flag
varieties admits a transparent interpretation in terms of Young
diagrams. Namely, let $\mathbf a$ and $\mathbf b$ be two
compositions of~$d$. Then, by Corollary~\ref{crl_partial_order}, the
condition ${\bl \Fl_{\mathbf a}(V) \br \preccurlyeq \bl \Fl_{\mathbf
b}(V) \br}$ is equivalent to $\mathbf a^\natural \succcurlyeq
\mathbf b^\natural$. According to the description of the partial
order on the set of partitions given in
\S\,\ref{subsec_nilpotent_orbits}, the latter can be interpreted as
follows: the diagram $\YD(\mathbf a^\natural)$ can be obtained from
the diagram $\YD(\mathbf b^\natural)$ by \textit{crumbling}, that
is, by moving boxes from upper rows to lower ones. For example,
using this interpretation it is easy to construct the Hasse diagram
for the partial order on the set of nil-equivalence classes of
$\FF^6$-flag varieties appearing in Figure~\ref{fig_diagrams}. This
diagram is shown in Figure~\ref{fig_Hasse_diagram}.

\begin{figure}[h]
\begin{center}

\begin{picture}(100,98)
\put(50,2){%
\put(-32,87){$\bl \Fl(1, 2; \FF^6) \br$}%
\put(-15,82){\line(-1,-1){10}}%
\put(15,82){\line(1,-1){10}}%
\put(-59,58){$\bl \Fl(1, 1; \FF^6) \br$}%
\put(9,58){$\bl \Gr_3(\FF^6) \br$}%
\put(-15,43){\line(-1,1){10}}%
\put(15,43){\line(1,1){10}}%
\put(-25,29){$\bl \Gr_2(\FF^6) \br$}%
\put(0,24){\line(0,-1){10}}%
\put(-19,00){$\bl \PP(\FF^6) \br$}%
}%
\end{picture}

\caption{} \label{fig_Hasse_diagram}

\end{center}
\end{figure}

\subsection{Necessary sphericity conditions for actions on
$V$-flag varieties}

Let $K \subset \GL(V)$ be a connected reductive subgroup. In this
subsection, making use of the description of the partial order on
the set $\mathscr F(\GL(V)) / \! \sim$ given
in~\S\S\,\ref{subsec_flags&nilpotent_orbits},~\ref{subsec_Young_diagrams},
we obtain two necessary conditions for $V$-flag varieties to be
$K$-spherical (see Propositions~\ref{prop_P(V)_is_spherical}
and~\ref{prop_Gr2_is_spherical}). These conditions will be starting
points in the proof of Theorem~\ref{thm_main_theorem}.

Let $\mathbf a$ be a nontrivial composition of~$d$.

\begin{proposition}[{see also~\cite[Theorem~5.8]{Pet}}]
\label{prop_P(V)_is_spherical} If the variety $\Fl_{\mathbf a}(V)$
is $K$-spherical, then so is the variety $\PP(V)$.
\end{proposition}

\begin{proof}
One has $\PP(V) = \Fl_{\mathbf b}(V)$ where $\mathbf b = (1, d -
1)$. As $\mathbf a \ne (d)$, one has $\mathbf b^\natural
\succcurlyeq \mathbf a^\natural$. In view of
Corollary~\ref{crl_partial_order} the latter implies that $\bl
\Fl_{\mathbf a}(V) \br \succcurlyeq \bl \PP(V) \br$. It remains to
apply Theorem~\ref{thm_prec_spherical}.
\end{proof}

\begin{corollary} \label{crl_V_is_spherical}
If\, $\Fl_{\mathbf a}(V)$ is a $K$-spherical variety, then $V$ is a
spherical $(K \times \FF^\times)$-module \textup(where $\FF^\times$
acts on $V$ by scalar transformations\textup).
\end{corollary}

\begin{proposition} \label{prop_Gr2_is_spherical}
If\, $\Fl_{\mathbf a}(V)$ is a $K$-spherical variety, $\bl
\Fl_{\mathbf a}(V) \br \ne \bl \PP(V) \br$, and $d \ge 4$, then
$\Gr_2(V)$ is a $K$-spherical variety.
\end{proposition}

\begin{proof}
We have $\Gr_2(V) = \Fl_{\mathbf b}(V)$, where $\mathbf b = (2, d -
2)$. As $d \ge 4$, we have $\mathbf b^\natural = (d - 2, 2)$. Next,
the hypothesis implies that the partition $\mathbf a^\natural$ is
different from $(d)$ and $(d - 1, 1)$. The latter means that
$\mathbf b^\natural \succcurlyeq \mathbf a^\natural$, whence by
Corollary~\ref{crl_partial_order} we obtain $\bl \Fl_{\mathbf a}(V)
\br \succcurlyeq \bl \Gr_2(V) \br$. The proof is completed by
applying Theorem~\ref{thm_prec_spherical}.
\end{proof}

\begin{remark}
For $d = 3$ one has $\Gr_2(V) \sim \PP(V)$.
\end{remark}

\section{Tools} \label{sect_tools}

In this section we collect all auxiliary results that will be needed
in our proof of Theorem~\ref{thm_main_theorem}.

\subsection{Spherical varieties and spherical subgroups}

Let $K$ be a connected reductive group, $B$ a Borel subgroup of~$K$,
and $X$ a spherical $K$-variety. The following proposition is well
known; we supply it with a proof for the reader's convenience.

\begin{proposition} \label{prop_inequality}
One has
\begin{equation} \label{eqn_inequality_main}
\dim K + \rk K \ge 2\dim X.
\end{equation}
\end{proposition}

\begin{proof}
Since there is an open $B$-orbit in~$X$, one has $\dim B \ge \dim
X$. To complete the proof it remains to notice that $2\dim B = \dim
K + \rk K$.
\end{proof}

In what follows we shall need the following notion. A subgroup $H
\subset K$ is said to be \textit{spherical} if the homogeneous space
$K / H$ is a spherical $K$-variety. It is easy to see that $H$ is
spherical if and only if $H^0$ is so.

\subsection{Homogeneous bundles}
\label{subsec_homogeneous_bundles}

Let $G$ be a group, $H$ a subgroup of $G$, and $X$ a $G$-variety.
Suppose that there is a surjective $G$-equivariant morphism $\varphi
\colon X \to G / H$. Let $Y$ denote the fiber of $\varphi$ over the
point $o = eH$. Evidently, $Y$ is an $H$-variety. In this situation
we say that $X$ is a \textit{homogeneous bundle over $G / H$ with
fiber~$Y$} (or simply a \textit{homogeneous bundle over $G / H$}).

Since $\varphi$ is $G$-equivariant, we have the following facts:

(1) every $G$-orbit in~$X$ meets~$Y$;

(2) for $g \in G$ and $y \in Y$ the condition $gy \in Y$ holds if
and only if $g \in H$.

Facts~(1) and~(2) imply that for every $G$-orbit $O \subset X$ the
intersection $O \cap Y$ is a nonempty $H$-orbit.

\begin{proposition} \label{prop_homogeneous_bundles_bijection}
The map $\iota \colon O \mapsto O \cap Y$ is a bijection between
$G$-orbits in~$X$ and $H$-orbits in~$Y$. Moreover, for every
$G$-orbit $O \subset X$ one has $$\dim O - \dim (O \cap Y) = \dim X
- \dim Y = \dim G / H.$$
\end{proposition}

\begin{proof}
It is easy to see that the map inverse to~$\iota$ takes an arbitrary
$H$-orbit $Y_0 \subset Y$ to the $G$-orbit $GY_0 \subset X$. Now
consider an arbitrary $G$-orbit $O \subset X$ and an arbitrary point
$y \in \nobreak O \cap Y$. Making use of fact~(2), we obtain $G_y
\subset H$, whence $O \simeq G / G_y$ and $O \cap Y \simeq H / G_y$.
Consequently, $\dim O - \dim (O \cap Y) = \dim G - \dim H$. Since
all fibers of $\varphi$ are isomorphic to~$Y$ (all of them are
$G$-shifts of~$Y$), we have $\dim X - \dim Y = \dim G / H$, hence
the required equalities.
\end{proof}

\begin{corollary} \label{crl_open_orbits}
There is an open $G$-orbit in $X$ if and only if there is an open
$H$-orbit in~$Y$.
\end{corollary}

\subsection{Supplementary information on $V$-flag varieties}

Let $V$ be a vector space of dimension~$d$ and let $\mathbf a =
(a_1, \ldots, a_s)$ be a nontrivial composition of~$d$. We put $m =
a_1 + \ldots + a_{s-1} = d - a_s$ and consider the vector space
$$
U = \underbrace{V \oplus \ldots \oplus V}_m \simeq V \otimes \FF^m.
$$
The natural $(\GL(V) \times \GL_m)$-module structure on $V \otimes
\FF^m$ is transferred to~$U$ so that $\GL(V)$ acts diagonally on $U$
and the action of $\GL_m$ on $U$ is given by the formula
$$
(g, (v_1, \ldots, v_m)) \mapsto (v_1, \ldots, v_m) g^\top,
$$
where $g \in \GL_m$ and $(v_1, \ldots, v_m) \in U$.

Consider the open subset $U_0 \subset U$ formed by all tuples $(v_1,
\ldots, v_m)$ of linearly independent vectors. Evidently, $U_0$ is a
$(\GL(V) \times \GL_m)$-stable subset. There is the natural
$\GL(V)$-equivariant surjective map
$$
\rho \colon U_0 \to \Fl_{\mathbf a}(V)
$$
taking a tuple $u = (v_1, \ldots, v_m) \in U_0$ to the tuple of
subspaces $\rho(u) = (V_1, \ldots, V_s)$, where $V_i = \langle v_1,
\ldots, v_{a_1 + \ldots + a_i} \rangle$ for $i = 1, \ldots, s-1$ and
$V_s = V$.

Let $e_1, \ldots, e_m$ be the standard basis in~$\FF^m$. We denote
by~$Q_{\mathbf a}$ the subgroup of~$\GL_m$ preserving each of the
subspaces
$$
\langle e_1, \ldots, e_{a_1 + \ldots + a_i} \rangle, \qquad i = 1,
\ldots, s-1.
$$
It is easy to see that the fibers of $\rho$ are exactly the orbits
of the group~$Q_{\mathbf a}$.

\begin{proposition} \label{prop_open_orbit_in_Fl}
Let $G \subset \GL(V)$ be an arbitrary subgroup and let $\mathbf a$
be a nontrivial composition of~$d$. The following conditions are
equivalent:

\textup{(a)} there is an open $G$-orbit in $\Fl_{\mathbf a}(V)$;

\textup{(b)} there is an open $(G \times Q_{\mathbf a})$-orbit in~$V
\otimes \mathbb F^m$.
\end{proposition}

\begin{proof}
It suffices to prove that the existence of an open $G$-orbit
in~$\Fl_{\mathbf a}$ is equivalent to the existence of an open $(G
\times Q_{\mathbf a})$-orbit in~$U_0$. The latter is implied by the
fact that the fibers of~$\rho$ are exactly the $Q_{\mathbf
a}$-orbits.
\end{proof}

\begin{corollary} \label{crl_Fl(1,1,...,1;V)}
Suppose that $K \subset \GL(V)$ is a connected reductive subgroup
and
$$
\mathbf a(m) = (\underbrace{1, \ldots, 1}_m, d - m),
$$ where $0 < m < d$. Then the following conditions are equivalent:

\textup{(a)} $\Fl_{\mathbf a(m)}(V)$ is a $K$-spherical variety;

\textup{(b)} $V \otimes \mathbb F^m$ is a spherical $(K \times
\GL_m)$-module.
\end{corollary}

\begin{proof}
Observe that $Q_{\mathbf a(m)}$ is nothing else than a Borel
subgroup of~$\GL_m$. It remains to apply
Proposition~\ref{prop_open_orbit_in_Fl} with~$G$ being a Borel
subgroup of~$K$.
\end{proof}

The following result is also a particular case
of~\cite[Lemma~1]{Ela72}.

\begin{corollary} \label{crl_grassmannians}
Let $0 < m < d$ and let $K \subset \GL(V)$ be an arbitrary subgroup.
Suppose that for the natural action of $K \times \GL_m$ on $V
\otimes \FF^m$ there is a point $z \in V \otimes \FF^m$ with
stabilizer~$H$ such that the orbit of~$z$ is open. Then for the
action of $K$ on $\Gr_m(V)$ the orbit of $\rho(z)$ is open and the
stabilizer $K_{\rho(z)}$ equals~$p(H)$, where $p \colon K \times
\GL_m \to K$ is the projection to the first factor. Moreover, $p(H)
\simeq H$.
\end{corollary}

\begin{remark}
It follows from what we have said in this subsection that for an
arbitrary nontrivial composition $\mathbf a$ of~$d$ the variety
$\Fl_{\mathbf a}(V)$ is the geometric quotient of $U_0$ by the
action of~$Q_{\mathbf a}$; see~\cite[\S\,4.2 and Theorem~4.2]{VP}.
\end{remark}

\subsection{Sphericity of some actions on Grassmannians}

The following particular cases of spherical actions on Grassmannians
are well known.

\begin{proposition} \label{prop_sympl_grassm}
For $n \ge 2$ and $1 \le m \le 2n - 1$, the action of $\Sp_{2n}$ on
$\Gr_m(\FF^{2n})$ is spherical.
\end{proposition}

\begin{proposition} \label{prop_orth_grassm}
For $n \ge 3$ and $1 \le m \le n - 1$, the action of $\SO_n$ on
$\Gr_m(\FF^n)$ is spherical.
\end{proposition}

Proofs of Propositions~\ref{prop_sympl_grassm}
and~\ref{prop_orth_grassm} can be found, for instance,
in~\cite[\S\,5.2]{HNOO} and~\cite[\S\,5.1]{HNOO}, respectively.

\subsection{A method for verifying sphericity of some actions}

Let $K$ be a connected reductive group, $B$ a Borel subgroup of~$K$,
and $X$ a spherical $K$-variety.

\begin{definition} \label{dfn_P-property}
We say that a point $x \in X$ and a connected reductive subgroup $L
\subset K_x$ \textit{have property}~(P) if the orbit $Kx$ is open
in~$X$ and for every pair $(Z, \varphi)$, where $Z$ is an
irreducible $K$-variety and $\varphi \colon Z \to X$ is a surjective
$K$-equivariant morphism with irreducible fiber over~$x$, the
following conditions are equivalent:

\textup{(1)} the variety $Z$ is $K$-spherical;

\textup{(2)} the variety $\varphi^{-1}(x) \subset Z$ is
$L$-spherical.
\end{definition}

\begin{proposition} \label{prop_sufficient_(P)}
Suppose that a point $x \in X$ and a connected reductive subgroup $L
\subset K_x$ are such that the orbit $B x$ is open in~$X$ and
$B_x^0$ is a Borel subgroup of~$L$. Then $x$ and $L$ have
property~\textup{(P)}.
\end{proposition}

\begin{proof}
Let $Z$ be an arbitrary irreducible $K$-variety and let $\varphi
\colon Z \to X$ be a surjective $K$-equivariant morphism with
irreducible fiber over~$x$. Since the orbit $B x$ is open in~$X$,
the set $Z_0 = \varphi^{-1}(Bx)$ is open in~$Z$ and is a homogeneous
bundle over $B x$. By Corollary~\ref{crl_open_orbits}, the existence
in~$Z_0$ of an open $B$-orbit is equivalent to the existence
in~$\varphi^{-1}(x)$ of an open $B_x$-orbit. As $B_x^0$ is a Borel
subgroup of~$L$, the latter condition is equivalent to sphericity of
the action $L : \varphi^{-1}(x)$.
\end{proof}

The following theorem is implied by results of Panyushev's
paper~\cite{Pan90}.

\begin{theorem} \label{thm_existence}
For every spherical $K$-variety $X$ there exist a point $x \in X$
and a connected reductive subgroup $L \subset K_x$ having
property~\textup{(P)}.
\end{theorem}

\begin{proof}
It follows from~\cite[Theorem~1]{Pan90} that there exist a point $x
\in X$ and a connected reductive subgroup $L \subset K_x$ such that
the orbit $Bx$ is open in~$X$ and $B_x^0$ is a Borel subgroup
of~$L$. Then by Proposition~\ref{prop_sufficient_(P)} the point~$x$
and the subgroup~$L$ have property~(P).
\end{proof}

In the remaining part of this subsection we find explicitly a point
$x \in X$ and a connected reductive subgroup $L \subset K_x$ having
property~(P) for the following two cases:

(1) $K = \SL_n$, $X = \Gr_m(\FF^n)$, where $n \ge 2$ and $1 \le m\le
n-1$;

(2) $K = \Sp_{2n}$, $X = \Gr_m(\FF^{2n})$, where $n \ge 2$ and $1
\le m \le 2n-1$.

\noindent The explicit form of $x$ and $L$ in these cases will be
many times used in~\S\,\ref{sect_proof_of_main_theorem}.
Proposition~\ref{prop_point+group_sl} corresponds to case~(1),
Proposition~\ref{prop_point+group_sp_even} and
Corollary~\ref{crl_point+group_sp_even} correspond to case~(2) with
$m = 2k$, Proposition~\ref{prop_point+group_sp_odd} and
Corollary~\ref{crl_point+group_sp_odd} correspond to case~(2) with
$m = 2k+1$.
Propositions~\ref{prop_point+group_sl},~\ref{prop_point+group_sp_even},
and~\ref{prop_point+group_sp_odd} are the most complicated
statements of this paper from the technical viewpoint, therefore we
postpone their proofs until Appendix~\ref{sect_appendix}.

\begin{proposition} \label{prop_point+group_sl}
Suppose that $n \ge 2$, $1 \le m \le n-1$, $V = \FF^n$, $K = \SL_n$,
and $X = \Gr_m(V)$. Then there are a point $[W] \in X$ and a
connected reductive subgroup $L \subset K_{[W]}$ satisfying the
following conditions:

\begin{enumerate}[label=\textup{(\arabic*)},ref=\textup{\arabic*}]
\item \label{4.13.1}
the point $[W]$ and the subgroup $L \subset K$ have
property~\textup{(P)};

\item \label{4.13.2}
$L \simeq \mathrm{S}(\mathrm{L}_m \times \mathrm{L}_{n-m})$;

\item
the pair $(L, W)$ is geometrically equivalent to the pair $(\GL_m,
\FF^m)$;

\item \label{4.13.4}
the pair $(L, V)$ is geometrically equivalent to the pair
$(\mathrm{S}(\mathrm{L}_m \times \mathrm{L}_{n-m}), \FF^n)$.
\end{enumerate}
\end{proposition}

\begin{proof}
See Appendix~\ref{sect_appendix}.
\end{proof}

\begin{proposition} \label{prop_point+group_sp_even}
Suppose that $n \ge 2$, $1 \le k \le n/2$, $V = \FF^{2n}$, $K =
\Sp_{2n}$, and $X = \Gr_{2k}(V)$. Then there are a point $[W] \in X$
and a connected reductive subgroup $L \subset K_{[W]}$ satisfying
the following conditions:

\begin{enumerate}[label=\textup{(\arabic*)},ref=\textup{\arabic*}]
\item \label{4.14.1}
the point $[W]$ and the subgroup $L \subset K$ have
property~\textup{(P)};

\item \label{4.14.2}
$L = L_1 \times \ldots \times L_k \times L_{k+1}$,
where $L_i \simeq \SL_2$ for $i = 1, \ldots, k$ and $L_{k+1} \simeq
\Sp_{2n-4k}$;

\item \label{4.14.3}
there is a decomposition $W = W_1 \oplus \ldots \oplus W_k$ into a
direct sum of $L$-modules so that:

\begin{enumerate}

\item[\textup{(3.1)}]
$\dim W_i = 2$ for $i = 1, \ldots, k$;

\item[\textup{(3.2)}]
for every $i = 1, \ldots, k$ the group $L_i$ acts trivially on all
summands~$W_j$ with $j \ne i$;

\item[\textup{(3.3)}]
for every $i = 1, \ldots, k$ the pair $(L_i, W_i)$ is geometrically
equivalent to the pair $(\SL_2, \FF^2)$;
\end{enumerate}

\item \label{4.14.4}
there is a decomposition $V = V_1 \oplus \ldots \oplus V_k
\oplus\nobreak V_{k+1}$ into a direct sum of $L$-modules so that:

\begin{enumerate}
\item[\textup{(4.1)}]
$\dim V_i = 4$ for $i = 1, \ldots, k$ and $\dim V_{k+1} = 2n - 4k$;

\item[\textup{(4.2)}]
for every $i = 1, \ldots, k+1$ the group $L_i$ acts trivially on all
summands~$V_j$ with $j \ne i$;

\item[\textup{(4.3)}]
for every $i = 1, \ldots, k$ the pair $(L_i, V_i)$ is geometrically
equivalent to the pair $(\SL_2, \FF^2 \oplus \FF^2)$, where $\SL_2$
acts diagonally, and the pair $(L_{k+1}, V_{k+1})$ is geometrically
equivalent to the pair $(\Sp_{2n-4k}, \FF^{2n-4k})$.
\end{enumerate}
\end{enumerate}

\textup{(}For $n = 2k$ the group $L_{k+1}$ and the space $V_{k+1}$
should be regarded as trivial.\textup{)}
\end{proposition}

\begin{proof}
See Appendix~\ref{sect_appendix}.
\end{proof}

\begin{corollary} \label{crl_point+group_sp_even}
In the hypotheses and notation of
Proposition~\textup{\ref{prop_point+group_sp_even}} suppose that a
point $[W] \in \Gr_{2k}(V)$ and a group $L \subset K_{[W]}$ satisfy
conditions \textup{(\ref{4.14.1})--(\ref{4.14.4})}. Then, for the
variety $\Gr_{2n-2k}(V)$, the point $[W^\perp]$ and the group $L$
have property~\textup{(P)}, and the pair $(L, W^\perp)$ is
geometrically equivalent to the pair $(L, W \oplus V_{k+1})$, where
$L$ acts diagonally.
\end{corollary}

\begin{proof}
Let $\Omega$ be a $K$-invariant symplectic form on~$V$. There is the
natural $K$-equivariant isomorphism $\Gr_{2k}(V) \simeq
\Gr_{2n-2k}(V)$ taking each $2k$-dimensional subspace of~$V$ to its
skew-orthogonal complement with respect to~$\Omega$. In view of
condition~(\ref{4.14.1}) this implies that the point $[W^\perp]$ and
the group $L$ have property~\textup{(P)}. Further, since the
$K$-orbit of $[W]$ is open in~$\Gr_{2k}(V)$, the restriction of
$\Omega$ to the subspace $W$ is nondegenerate. Hence $V = W \oplus
W^\perp$, which by conditions~(\ref{4.14.3}) and~(\ref{4.14.4})
uniquely determines the $L$-module structure on~$W^\perp$.
\end{proof}

\begin{proposition} \label{prop_point+group_sp_odd}
Suppose that $n \ge 2$, $0 \le k \le (n-1)/2$, $V = \FF^{2n}$, $K =
\Sp_{2n}$, and $X = \Gr_{2k+1}(V)$. Then there are a point $[W] \in
X$ and a connected reductive subgroup $L \subset K_{[W]}$ satisfying
the following conditions:

\begin{enumerate}[label=\textup{(\arabic*)},ref=\textup{\arabic*}]
\item \label{4.16.1}
the point $[W]$ and the subgroup $L \subset K$ have
property~\textup{(P)};

\item \label{4.16.2}
$L = L_0 \times L_1 \times \ldots \times L_k \times L_{k+1}$, where
$L_0 \simeq \FF^\times$, $L_i \simeq \SL_2$ for $i = 1, \ldots, k$
and $L_{k+1} \simeq \Sp_{2n-4k-2}$;

\item \label{4.16.3}
there is a decomposition $W = W_0 \oplus W_1 \oplus \ldots \oplus
W_k$ into a direct sum of $L$-modules so that:

\begin{enumerate}
\item[\textup{(3.1)}]
$\dim W_0 = 1$ and $\dim W_i = 2$ for $i = 1, \ldots, k$;

\item[\textup{(3.2)}]
for every $i = 0, 1, \ldots, k$ the group $L_i$ acts trivially on
all summands~$W_j$ with $j \ne i$;

\item[\textup{(3.3)}]
the pair $(L_0, W_0)$ is geometrically equivalent to the pair
$(\FF^\times, \FF^1)$ and for every $i = 1, \ldots, k$ the pair
$(L_i, W_i)$ is geometrically equivalent to the pair $(\SL_2,
\FF^2)$;
\end{enumerate}

\item \label{4.16.4}
there is a decomposition $V = V_0 \oplus V_1 \oplus \ldots \oplus
V_k \oplus\nobreak V_{k+1}$ into a direct sum of $L$-modules so
that:

\begin{enumerate}
\item[\textup{(4.1)}]
$\dim V_0 = 2$, $\dim V_i = 4$ for $i = 1, \ldots, k$, and $\dim
V_{k+1} = 2n - 4k - 2$;

\item[\textup{(4.2)}]
for every $i = 0, 1, \ldots, k+1$ the group $L_i$ acts trivially on
all summands~$V_j$ with $j \ne i$;

\item[\textup{(4.3)}]
the pair $(L_0, V_0)$ is geometrically equivalent to the pair
$(\FF^\times, \FF^1 \oplus \FF^1)$ with the action $(t,(x_1, x_2))
\mapsto (tx_1, t^{-1}x_2)$, for every $i = 1, \ldots, k$ the pair
$(L_i, V_i)$ is geometrically equivalent to the pair $(\SL_2, \FF^2
\oplus\nobreak \FF^2)$ with $\SL_2$ acting diagonally, and the pair
$(L_{k+1}, V_{k+1})$ is geometrically equivalent to the pair
$(\Sp_{2n-4k-2}, \FF^{2n-4k-2})$.
\end{enumerate}
\end{enumerate}

\textup{(}For $n = 2k+1$ the group $L_{k+1}$ and the space $V_{k+1}$
should be regarded as trivial.\textup{)}

\end{proposition}

\begin{proof}
See Appendix~\ref{sect_appendix}.
\end{proof}

\begin{corollary} \label{crl_point+group_sp_odd}
Under the hypotheses and notation of
Proposition~\textup{\ref{prop_point+group_sp_odd}} suppose that a
point $[W] \in \Gr_{2k+1}(V)$ and a group $L \subset K_{[W]}$
satisfy conditions \textup{(\ref{4.16.1})--(\ref{4.16.4})}. Then,
for the variety $\Gr_{2n-2k-1}(V)$, the point $[W^\perp]$ and the
group $L$ have property~\textup{(P)} and the pair $(L, W^\perp)$ is
geometrically equivalent to the pair $(L, W \oplus V_{k+1})$, where
$L$ acts diagonally.
\end{corollary}

\begin{proof}
Let $\Omega$ be a $K$-invariant symplectic form on~$V$. There is a
natural $K$-equivariant isomorphism $\Gr_{2k+1}(V) \simeq
\Gr_{2n-2k-1}(V)$ taking each $(2k+1)$-dimensional subspace of~$V$
to its skew-orthogonal complement with respect to~$\Omega$. In view
of condition~(\ref{4.16.1}) this implies that the point $[W^\perp]$
and the group $L$ have property~\textup{(P)}. Further, since the
$K$-orbit of $[W]$ is open in $\Gr_{2k+1}(V)$, the restriction of
$\Omega$ to the subspace $W$ has rank~$2k$. Hence $\dim(W \cap
W^\perp) = 1$, which by condition~(\ref{4.16.4}) implies $W \cap
W^\perp = W_0$. Now the $L$-module structure on~$W^\perp$ is
uniquely determined by conditions~(\ref{4.16.3}) and~(\ref{4.16.4}).
\end{proof}

\subsection{Some sphericity conditions for actions on $V$-flag
varieties}

Let $K$ be a connected reductive group, $V$ a $K$-module, and $V =
V_1 \oplus V_2$ a decomposition of $V$ into a direct sum of two (not
necessarily simple) nontrivial $K$-submodules.

\begin{proposition} \label{prop_two_grassm_sphericity}
Let $1 \le k \le \dim V_1$, $Z = \Gr_k(V)$, and $X = \Gr_k(V_1)$.

\textup{(a)} Suppose that $Z$ is a $K$-spherical variety. Then $X$
is also a $K$-spherical variety.

\textup{(b)} Suppose that $X$ is a $K$-spherical variety. Suppose
that a point $[W_0] \in X$ and a connected reductive subgroup $L
\subset K_{[W_0]}$ have property~\textup{(P)}. Then the following
conditions are equivalent:

\textup{(1)} $Z$ is a $K$-spherical variety;

\textup{(2)} $W_0^* \otimes V_2$ is a spherical $L$-module.
\end{proposition}

\begin{proof}
Let $p$ denote the projection of $V$ to~$V_1$ along~$V_2$. Let $Z_0
\subset Z$ be the open $K$-stable subset consisting of all points
$[U]$ with $\dim p(U) = k$. The projection $p$ induces a surjective
$K$-equivariant morphism $\varphi \colon Z_0 \to\nobreak X$. For
each point $[W] \in X$ the fiber $\varphi^{-1}([W])$ consists of all
points~$[U]$ with $p(U) = W$, whence
$$
\varphi^{-1}([W]) \simeq \Hom (W, V_2) \simeq W^* \otimes V_2.
$$
It is easy to see that the $K$-sphericity of $Z_0$ implies that
of~$X$, which proves part~(a). To complete the proof of part~(b), it
remains to make use of property~(P) for the point $[W_0]$ and the
group~$L$.
\end{proof}

\begin{remark}
It is easy to see that $X$ is realized as a $K$-stable subvariety
of~$Z$. Then Proposition~\ref{prop_two_grassm_sphericity}(а) is also
implied by the following well-known fact: every irreducible
$K$-stable subvariety of a spherical $K$-variety is spherical (see,
for instance,~\cite[Proposition~15.14]{Tim}).
\end{remark}

\begin{corollary} \label{crl_k=dimV1}
Suppose that $k = \dim V_1$ and $Z = \Gr_k(V)$. Then the following
conditions are equivalent:

\textup{(1)} $Z$ is a $K$-spherical variety;

\textup{(2)} $V_1^* \otimes V_2$ is a spherical $K$-module.
\end{corollary}

\begin{proof}
In this situation $X$ consists of the single point $[V_1]$.
Evidently, this point and the group $K$ have property~(P).
\end{proof}

\begin{proposition} \label{prop_gr2_necessary}
Suppose that $\dim V \ge 4$ and $Z = \Gr_2(V)$ is a $K$-spherical
variety. Then $V_2 \otimes \FF^2$ is a spherical $(K \times
\GL_2)$-module.
\end{proposition}

\begin{proof}
We first consider the case $\dim V_1 \ge 2$. Put $X = \Gr_2(V_1)$.
Proposition~\ref{prop_two_grassm_sphericity}(a) yields that $X$ is a
spherical $K$-variety. By Theorem~\ref{thm_existence} there are a
point $x \in X$ and a connected reductive subgroup~$L \subset K_x$
having property~(P). Then
Proposition~\ref{prop_two_grassm_sphericity}(b) implies that $V_2
\otimes \FF^2$ is a spherical $L$-module, hence a spherical $(K
\times\nobreak \GL_2)$-module.

We now consider the case $\dim V_1 = 1$. Then $\dim V_2 \ge 3$. For
each two-dimensional subspace $W \subset V$ put $W_1 = W \cap V_2$
and let $W_2$ denote the projection of $W$ to~$V_2$ along~$V_1$.

Let $Z_0 \subset Z$ be the open $K$-stable subset consisting of all
points $[W]$ with $\dim W_1 =\nobreak 1$ and $\dim W_2 = 2$. It is
easy to see that the morphism
$$
Z_0 \to \Fl(1,1; V_2), \quad [W] \mapsto (W_1, W_2, V_2),
$$
is surjective and $K$-equivariant, hence the $K$-sphericity of
$\Gr_2(V)$ implies that of $\Fl(1,1; V_2)$. Then
Corollary~\ref{crl_Fl(1,1,...,1;V)} implies that $V_2 \otimes \FF^2$
is a spherical $(K \times \GL_2)$-module.
\end{proof}

\begin{proposition} \label{prop_Fl(r1,r2;V)}
Let $k_1 \ge 1$, $k_2 \ge 1$, $k_1 + k_2 \le \dim V_1$. Put $Z =
\Fl(k_1,k_2; V)$, $X = \Gr_{k_1+k_2}(V_1)$ and suppose that $X$ is a
$K$-spherical variety. Suppose that a point $[W_0] \in X$ and a
connected reductive subgroup $L \subset K_{[W_0]}$ have
property~\textup{(P)}. Then the following conditions are equivalent:

\textup{(1)} $Z$ is a $K$-spherical variety;

\textup{(2)} $(W_0^* \otimes V_2) \times \Gr_{k_1}(W_0)$ is an
$L$-spherical variety \textup($L$ acts diagonally\textup).
\end{proposition}

\begin{proof}
Let $p$ denote the projection of $V$ to~$V_1$ along~$V_2$. Let $Z_0
\subset Z$ be the open $K$-stable subset consisting of all points
$(U_1, U_2, V) \in Z$ with $\dim p(U_2) = k_1 + k_2$. Then $p$
induces the surjective $K$-equivariant morphism $\varphi \colon Z_0
\to \nobreak X$ taking a point $(U_1, U_2, V)$ to~$p(U_2)$. For
every point $[W] \in \nobreak X$, the fiber $\varphi^{-1}([W])$ is
isomorphic to
$$
\Hom(W, V_2) \times \Gr_{k_1}(W) \simeq (W^* \otimes\nobreak V_2)
\times \Gr_{k_1}(W).
$$
To complete the proof it remains to make use of property~(P) for
$[W_0]$ and~$L$.
\end{proof}

\section{Known classifications used in the paper}
\label{sect_known_classifications}

\subsection{Classification of spherical modules}
\label{subsec_spherical_modules}

In this subsection we present the classification of spherical
modules obtained in the papers~\cite{Kac},~\cite{BR},
and~\cite{Lea}.

Let $K$ be a connected reductive group and let $C$ be the connected
component of the identity of the center of~$K$. For every simple
$K$-module $V$ we consider the character $\chi \in \mathfrak X(C)$
via which $C$ acts on~$V$.

\begin{theorem}{\cite[Theorem~3]{Kac}}
\label{thm_spherical_modules_simple} A simple $K$-module $V$ is
spherical if and only if the following conditions hold:

\textup{(1)} up to a geometrical equivalence, the pair $(K', V)$ is
contained in Table~\textup{\ref{table_spherical_modules}};

\textup{(2)} the group $C$ satisfies the conditions listed in the
fourth column of Table~\textup{\ref{table_spherical_modules}}.
\end{theorem}

\begin{table}[h]

\caption{} \label{table_spherical_modules}

\begin{center}

\begin{tabular}{|c|c|c|c|c|}
\hline

No. & $K'$ & $V$ & Conditions on $C$ & Note \\

\hline

1 & $\SL_n$ & $\FF^n$ & $\chi \ne 0$ for $n = 1$ & $n \ge 1$ \\

\hline

2 & $\SO_n$ & $\FF^n$ & $\chi \ne 0$ & $n \ge 3$ \\

\hline

3 & $\Sp_{2n}$ & $\FF^{2n}$ & & $n \ge 2$ \\

\hline

4 & $\SL_n$ & $\mathrm S^2 \FF^n$ & $\chi \ne 0$ & $n \ge 3$ \\

\hline

5 & $\SL_n$ & $\wedge^2 \FF^n$ & $\chi \ne 0$ for $n = 2k$ &
$n \ge 5$ \\

\hline

6 & $\SL_n \times \SL_m$ & $\FF^n \otimes \FF^m$ & $\chi \ne 0$ for
$n = m$ & \begin{tabular}{c} $n, m \ge 2$ \\ $n + m \ge 5$
\end{tabular} \\

\hline

7 & $\SL_2 \times \Sp_{2n}$ & $\FF^2 \otimes \FF^{2n}$ & $\chi \ne
0$ & $n \ge 2$
\\

\hline

8 & $\SL_3 \times \Sp_{2n}$ & $\FF^3 \otimes \FF^{2n}$ & $\chi \ne
0$ & $n \ge 2$
\\

\hline

9 & $\SL_n \times \Sp_4$ & $\FF^n \otimes \FF^4$ & $\chi \ne 0$ for
$n = 4$ & $n \ge 4$ \\

\hline

10 & $\Spin_7$ & $\FF^8$ & $\chi \ne 0$ & \\

\hline

11 & $\Spin_9$ & $\FF^{16}$ & $\chi \ne 0$ & \\

\hline

12 & $\Spin_{10}$ & $\FF^{16}$ & & \\

\hline

13 & $\mathsf G_2$ & $\FF^7$ & $\chi \ne 0$ & \\

\hline

14 & $\mathsf E_6$ & $\FF^{27}$ & $\chi \ne 0$ & \\

\hline
\end{tabular}

\end{center}

\end{table}

Let us give some comments and explanations for
Table~\ref{table_spherical_modules}. In rows 3--8 the restrictions
in the column ``Note'' are imposed in order to avoid coincidences
(up to a geometric equivalence) of the respective $K'$-modules with
$K'$-modules corresponding to other rows. In rows~10 and~11, the
group $K'$ acts on $V$ via the spin representation. In row~12, the
group $K'$ acts on $V$ via a (any of the two) half-spin
representation. At last, in rows~13 and~14 the group $K'$ acts on
$V$ via a faithful representation of minimal dimension.

Let $V$ be a $K$-module. We say that $V$ is \textit{decomposable} if
there exist connected reductive subgroups $K_1, K_2$, a
$K_1$-module~$V_1$, and a $K_2$-module~$V_2$ such that the pair $(K,
V)$ is geometrically equivalent to the pair $(K_1 \times K_2, V_1
\oplus V_2)$. We note that in this situation $V_1 \oplus\nobreak
V_2$ is a spherical $(K_1 \times K_2)$-module if and only if $V_1$
is a spherical $K_1$-module and $V_2$ is a spherical $K_2$-module.
We say that $V$ is \textit{indecomposable} if $V$ is not
decomposable. At last, we say that $V$ is \textit{strictly
indecomposable} if $V$ is an indecomposable $K'$-module. Evidently,
every simple $K$-module is strictly indecomposable.

Let $V$ be a $K$-module and let $V = V_1 \oplus \ldots \oplus V_r$
be a decomposition of~$V$ into a direct sum of simple
$K$-submodules. For every $i = 1, \ldots, r$ we denote by~$\chi_i$
the character of~$C$ via which~$C$ acts on~$V_i$.

\begin{theorem}[{\cite{BR}, \cite{Lea}}]
\label{thm_spherical_modules_nonsimple} In the above notation
suppose that $V$ is a nonsimple $K$-module. Then $V$ is a strictly
indecomposable spherical $K$-module if and only if $r = 2$ and the
following conditions hold:

\textup{(1)} up to a geometrical equivalence, the pair $(K', V)$ is
contained in Table~\textup{\ref{table_spherical_modules_2}};

\textup{(2)} the group $C$ satisfies the conditions listed in the
third column of Table~\textup{\ref{table_spherical_modules_2}}.
\end{theorem}

\begin{table}

\begin{center}

\caption{} \label{table_spherical_modules_2}

\begin{tabular}{|c|c|c|c|}
\hline

No. & $(K',V)$ & Conditions on $C$ & Note \\

\hline

1 &

\begin{tabular}{c}
\begin{picture}(64,37)
\put(26,22){\line(-1,-2){5}}%
\put(36,22){\line(1,-2){5}}%
\put(12,0){$\FF^n \oplus \FF^n$}%
\put(24,25){$\SL_n$}%
\end{picture}
\end{tabular}

& \begin{tabular}{c}
$\chi_1, \chi_2$ lin. ind. for $n = 2$;\\
\hline $\chi_1 \ne \chi_2$ for $n \ge 3$
\end{tabular}
& $n \ge 2$ \\

\hline

2 &

\begin{tabular}{c}
\begin{picture}(64,37)
\put(27,22){\line(-1,-1){10}}%
\put(35,22){\line(1,-1){10}}%
\put(6,0){$(\FF^n\!)^* \oplus \FF^n$}%
\put(24,25){$\SL_n$}%
\end{picture}
\end{tabular}

& $\chi_1 \ne - \chi_2$ & $n \ge 3$ \\

\hline

3 &

\begin{tabular}{c}
\begin{picture}(64,37)
\put(27,22){\line(-1,-1){10}}%
\put(35,22){\line(1,-1){10}}%
\put(6,0){$\FF^n \oplus \wedge^2 \FF^n$}%
\put(24,25){$\SL_n$}%
\end{picture}
\end{tabular}

& \begin{tabular}{c} $\chi_2 \ne 0$ for $n = 2k$
\\ \hline $\chi_1 \ne - \frac{n-1}{2} \chi_2$
for $n = 2k{+}1$ \\ \end{tabular}
& $n \ge 4$ \\

\hline

4 &

\begin{tabular}{c}
\begin{picture}(64,37)
\put(27,22){\line(-1,-1){10}}%
\put(35,22){\line(1,-1){10}}%
\put(0,0){$(\FF^n\!)^* \oplus \wedge^2 \FF^n$}%
\put(24,25){$\SL_n$}%
\end{picture}
\end{tabular}

& \begin{tabular}{c} $\chi_2 \ne 0$ for $n = 2k$ \\ \hline
$\chi_1 \ne \frac{n-1}{2} \chi_2$ for $n = 2k{+}1$ \\
\end{tabular} & $n \ge 4$ \\

\hline

5 &

\begin{tabular}{c}
\begin{picture}(76,37)
\put(13,22){\line(-1,-2){5}}%
\put(22,22){\line(1,-1){10}}%
\put(52,22){\line(1,-1){10}}%
\put(0,0){$\FF^n \oplus (\FF^n \otimes \FF^m\!)$}%
\put(11,25){$\SL_n \times \SL_m$}%
\end{picture}
\end{tabular}

& \begin{tabular}{c} $\chi_1 \ne 0$ for $n \le m - 1$ \\ \hline
$\chi_1, \chi_2$ lin. ind. \\ for $n = m, m + 1$ \\ \hline $\chi_1
\ne \chi_2$ for $n \ge m + 2$ \\ \end{tabular}
& $n,m \ge 2$ \\

\hline

6 &

\begin{tabular}{c}
\begin{picture}(88,37)
\put(19,22){\line(-1,-2){5}}%
\put(29,22){\line(3,-2){15}}%
\put(57,22){\line(3,-2){15}}%
\put(0,0){$(\FF^n\!)^* \oplus (\FF^n \otimes \FF^m\!)$}%
\put(17,25){$\SL_n \times \SL_m$}%
\end{picture}
\end{tabular}

& \begin{tabular}{c} $\chi_1 \ne 0$ for $n \le m - 1$ \\ \hline
$\chi_1, \chi_2$ lin. ind. \\ for $n = m, m + 1$ \\ \hline $\chi_1
\ne - \chi_2$ for $n \ge m + 2$ \\ \end{tabular}
& $n \ge 3, m \ge 2$ \\

\hline

7 &

\begin{tabular}{c}
\begin{picture}(76,37)
\put(14,22){\line(-1,-2){5}}%
\put(22,22){\line(1,-1){10}}%
\put(47,22){\line(3,-2){15}}%
\put(0,0){$\FF^2 \oplus (\FF^2 \otimes \FF^{2n}\!)$}%
\put(11,25){$\SL_2 \times \Sp_{2n}$}%
\end{picture}
\end{tabular}

& $\chi_1, \chi_2$ lin. ind. & $n \ge 2$ \\

\hline

8 &

\begin{tabular}{c}
\begin{picture}(110,37)
\put(19,22){\line(-1,-1){10}}%
\put(48,22){\line(-1,-1){10}}%
\put(56,22){\line(1,-1){10}}%
\put(84,22){\line(1,-1){10}}%
\put(0,0){$(\FF^n \otimes \FF^2\!) \oplus (\FF^2 \otimes \FF^m\!)$}%
\put(12,25){$\SL_n \times \SL_2 \times \SL_m$}%
\end{picture}
\end{tabular}

& \begin{tabular}{c} $\chi_1, \chi_2$ lin. ind. \\ for $n = m = 2$
\\ \hline $\chi_2 \ne 0$ for $n \ge 3$ and $m = 2$ \\
\end{tabular}
& $n \ge m \ge 2$ \\

\hline

9 &

\begin{tabular}{c}
\begin{picture}(114,37)
\put(19,22){\line(-1,-1){10}}%
\put(48,22){\line(-1,-1){10}}%
\put(56,22){\line(1,-1){10}}%
\put(82,22){\line(1,-1){10}}%
\put(0,0){$(\FF^n \otimes \FF^2\!) \oplus (\FF^2 \otimes \FF^{2m}\!)$}%
\put(12,25){$\SL_n \times \SL_2 \times \Sp_{2m}$}%
\end{picture}
\end{tabular}

& \begin{tabular}{c} $\chi_1, \chi_2$ lin. ind. for $n = 2$ \\
\hline $\chi_2 \ne 0$ for $n \ge 3$ \\ \end{tabular}
& $n,m \ge 2$ \\

\hline

10 &

\begin{tabular}{c}
\begin{picture}(118,37)
\put(16,22){\line(-1,-2){5}}%
\put(51,22){\line(-1,-1){10}}%
\put(60,22){\line(1,-1){10}}%
\put(84,22){\line(3,-2){15}}%
\put(0,0){$(\FF^{2n} \otimes \FF^2\!) \oplus
(\FF^2 \otimes \FF^{2m}\!)$}%
\put(12,25){$\Sp_{2n} \times \SL_2 \times \Sp_{2m}$}%
\end{picture}
\end{tabular}

& $\chi_1, \chi_2$ lin. ind. & $n \ge m \ge 2$ \\

\hline

11 &

\begin{tabular}{c}
\begin{picture}(48,37)
\put(18,22){\line(-1,-1){10}}%
\put(24,22){\line(1,-1){10}}%
\put(0,0){$\FF^{2n} \oplus \FF^{2n}$}%
\put(13,25){$\Sp_{2n}$}%
\end{picture}
\end{tabular}

& $\chi_1, \chi_2$ lin. ind. & $n \ge 2$ \\

\hline

12 &

\begin{tabular}{c}
\begin{picture}(42,37)
\put(16,22){\line(-1,-1){10}}%
\put(23,22){\line(1,-1){10}}%
\put(0,0){$\mathbb F^8_+\oplus \mathbb F^8_-$}%
\put(7,25){$\Spin_8$}%
\end{picture}
\end{tabular}

& $\chi_1, \chi_2$
lin. ind. & \\

\hline

\end{tabular}

\end{center}
\end{table}

Let us explain some notation in
Table~\ref{table_spherical_modules_2}. In the second column the pair
$(K', V)$ is arranged in two levels, with $K'$ in the upper level
and $V$ in the lower one. Further, each factor of $K'$ acts
diagonally on all components of~$V$ with which it is connected by an
edge. In row~12 the symbols $\FF^8_\pm$ stand for the spaces of the
two half-spin representations of $\Spin_8$.

Let $V$ be a $K$-module such that, up to a geometrical equivalence,
the pair $(K', V)$ is contained in one of
Tables~\ref{table_spherical_modules}
or~\ref{table_spherical_modules_2}. Using the information in the
column ``Conditions on~$C$'', to~$V$ we assign a multiset (that is,
a set whose members are considered together with their
multiplicities) $I(V)$ consisting of several characters of~$C$ in
the following way:

(1) $I(V) = \varnothing$ if there are no conditions on~$C$;

(2) $I(V) = \lbrace \psi_1 - \psi_2 \rbrace$ if the condition on~$C$
is of the form ``$\psi_1 \ne \psi_2$'' for some $\psi_1, \psi_2 \in
\mathfrak X(C)$;

(3) $I(V) = \lbrace \chi_1, \chi_2 \rbrace$ if the condition on~$C$
is of the form ``$\chi_1, \chi_2$ lin. ind.''

In the above notation, $V$ is a spherical $K$-module if and only if
all characters in~$I(V)$ are linearly independent in~$\mathfrak
X(C)$.

Let $V$ be an arbitrary $K$-module. It is easy to see that there are
a decomposition $V = W_1 \oplus \ldots \oplus W_p$ into a direct sum
of $K$-submodules (not necessarily simple) and connected semisimple
normal subgroups $K_1, \ldots, K_p \subset K'$ (some of them are
allowed to be trivial) with the following properties:

(1) $W_i$ is a strictly indecomposable $K$-module for all $i = 1,
\ldots, p$;

(2) the pair $(K', V)$ is geometrically equivalent to the pair
$$
(K_1 \times \ldots \times K_p, W_1 \oplus \ldots \oplus W_p).
$$

The theorem below provides a sphericity criterion for the
$K$-module~$V$. This theorem is a reformulation
of~\cite[Theorem~7]{BR}; see also~\cite[Theorem~2.6]{Lea}.

\begin{theorem} \label{thm_spherical_modules_general}
In the above notation, $V$ is a spherical $K$-module if and only if
the following conditions hold:

\textup{(1)} $W_i$ is a spherical $K$-module for all $i = 1, \ldots,
p$;

\textup{(2)} all the $|I(W_1)| + \ldots + |I(W_p)|$ characters in
the multiset $I(W_1) \cup \ldots \cup I(W_p)$ are linearly
independent in~$\mathfrak X(C)$.
\end{theorem}

\subsection{Classification of Levi subgroups in~$\GL(V)$ acting
spherically on $V$-flag varieties} \label{subsec_Levi_subgroups}

Let $G$ be an arbitrary connected reductive group. Let $P,Q \subset
G$ be parabolic subgroups and let $K$ be a Levi subgroup of~$P$.

The following lemma is known to specialists; for the reader's
convenience we provide it together with a proof.

\begin{lemma} \label{lemma_two_parabolics}
The following conditions are equivalent:

\textup{(a)} $G / Q$ is a $K$-spherical variety;

\textup{(b)} $G / P \times G / Q$ is a spherical variety with
respect to the diagonal action of~$G$.
\end{lemma}

\begin{proof}
It is well known that there is a Borel subgroup $B \subset G$ with
the following properties:

(1) the set $BP$ is open in~$G$;

(2) the group $B_K = B \cap P$ is a Borel subgroup of~$K$.

\noindent For the action $B : G / P$, the group $B_K$ is exactly the
stabilizer of the point $o = eP$ and the orbit $O = Bo \simeq B /
B_K$ is open. Consider the open subset $O \times G / Q$ in $G / P
\times G / Q$. This subset is a homogeneous bundle over $B / B_K$
with fiber $G / Q$, see \S\,\ref{subsec_homogeneous_bundles}.
Applying Corollary~\ref{crl_open_orbits} we find that the existence
of an open $B$-orbit in $O \times G / Q$ is equivalent to the
existence of an open $B_K$-orbit in $G / Q$, which implies the
required result.
\end{proof}

As was already mentioned in Introduction, there is a complete
classification of all $G$-spherical varieties of the form $G / P
\times G / Q$. Below, using Lemma~\ref{lemma_two_parabolics}, we
reformulate the results of this classification in the case $G =
\GL(V)$ and thereby list all cases where a Levi subgroup $K \subset
\GL(V)$ acts spherically on a $V$-flag variety $\Fl_{\mathbf a}(V)$
(see Theorem~\ref{thm_MWZ}).

Let $V$ be a vector space of dimension~$d$ and let $\mathbf d =
(d_1, \ldots, d_r)$ be a nontrivial composition of~$d$ such that
$d_1 \le \ldots \le d_r$. Fix a decomposition
$$
V = V_1 \oplus \ldots \oplus V_r
$$
into a direct sum of subspaces, where $\dim V_i = d_i$ for all $i =
1, \ldots, r$. Put
$$
K_{\mathbf d} = \GL(V_1) \times \ldots \times \GL(V_r) \subset
\GL(V).
$$
Let~$Q$ denote the stabilizer in $\GL(V)$ of the point
$$
(V_1, V_1 \oplus V_2, \ldots, V_1 \oplus \ldots \oplus V_r) \in
\Fl_{\mathbf d}(V).
$$
Clearly, $Q$ is a parabolic subgroup of $\GL(V)$ and $K_{\mathbf d}$
is a Levi subgroup of~$Q$. It is well known that every Levi subgroup
of~$\GL(V)$ is conjugate to a subgroup of the form~$K_{\mathbf d}$.

The following theorem follows from results of the paper~\cite{MWZ1},
see also~\cite[Corollary~1.3.A]{Stem}.

\begin{theorem} \label{thm_MWZ}
Suppose that $\mathbf a = (a_1, \ldots, a_s)$ is a nontrivial
composition of~$d$ such that $a_1 \le \ldots \le a_s$. Then the
variety $\Fl_{\mathbf a}(V)$ is $K_{\mathbf d}$-spherical if and
only if the pair of compositions $(\mathbf d, \mathbf a)$ is
contained in Table~\textup{\ref{table_MWZ}}.
\end{theorem}

\begin{table}[h]

\begin{center}

\caption{} \label{table_MWZ}

\begin{tabular}{|c|c|c|}

\hline
No. & $\mathbf d$ & $\mathbf a$ \\

\hline

1 & $(d_1, d_2)$ & $(a_1, a_2)$ \\

\hline

2 & $(2, d_2)$ & $(a_1, a_2, a_3)$ \\

\hline

3 & $(d_1, d_2, d_3)$ & $(2, a_2)$ \\

\hline

4 & $(d_1, d_2)$ & $(1, a_2, a_3)$ \\

\hline

5 & $(1, d_2, d_3)$ & $(a_1, a_2)$ \\

\hline

6 & $(1, d_2)$ & $(a_1, \ldots, a_s)$ \\

\hline

7 & $(d_1, \ldots, d_r)$ & $(1, a_2)$ \\

\hline
\end{tabular}
\end{center}
\end{table}

\section{Proof of theorem~\ref{thm_main_theorem}}
\label{sect_proof_of_main_theorem}

We divide the proof of Theorem~\ref{thm_main_theorem} into several
steps. As follows from Proposition~\ref{prop_Gr2_is_spherical}, the
first step of the proof is a description of all spherical actions on
the variety $\Gr_2(V)$. This is done in
\S\,\ref{subsec_Gr2_simple_V} (in the case where $V$ is a simple
$K$-module) and~\S\,\ref{subsec_Gr2_non-simple_V} (in the case where
$V$ is a nonsimple $K$-module). At the next step we classify all
spherical actions on arbitrary Grassmannians,
see~\S\,\ref{subsec_grassmannians}. Finally,
in~\S\,\ref{subsec_finish} we list all spherical actions on $V$-flag
varieties that are not Grassmannians.

We recall that the statement of Theorem~\ref{thm_main_theorem}
includes the following objects:

$V$ is a vector space of dimension~$d$;

$K$ is a connected reductive subgroup of~$\GL(V)$;

$C$ is the connected component of the identity of the center of~$K$.

Next, there is a decomposition
\begin{equation} \label{eqn_decomposition_of_V}
V = V_1 \oplus \ldots \oplus V_r
\end{equation}
into a direct sum of simple $K$-submodules and for every $i = 1,
\ldots, r$ the group $C$ acts on~$V_i$ via a character denoted
by~$\chi_i$.

We fix all the above-mentioned objects and notation until the end of
this section.

\subsection{Spherical actions on $\Gr_2(V)$ in the case where
$V$ is a simple $K$-module} \label{subsec_Gr2_simple_V}

The goal of this subsection is to prove the following theorem.

\begin{theorem} \label{thm_gr2_case_of_simple_V}
Suppose that $d \ge 4$ and $V$ is a simple $K$-module. Then the
variety $\Gr_2(V)$ is $K$-spherical if and only if, up to a
geometrical equivalence, the pair $(K', V)$ is contained in
Table~\textup{\ref{table_Gr2_simple_V}}.
\end{theorem}

\begin{table}[!h]
\begin{center}

\caption{} \label{table_Gr2_simple_V}

\begin{tabular}{|c|c|c|c|}
\hline

No. & $K'$ & $V$ & Note \\

\hline

1 & $\SL_n$ & $\FF^n$ & $n \ge 4$ \\

\hline

2 & $\Sp_{2n}$ & $\FF^{2n}$ & $n \ge 2$ \\

\hline

3 & $\SO_n$ & $\FF^n$ & $n \ge 4$ \\

\hline

4 & $\Spin_7$ & $\FF^8$ & \\

\hline
\end{tabular}
\end{center}
\end{table}

\begin{proof}
First of all we note that, since $V$ is a simple $K$-module, the
center of $K$ acts trivially on $\Gr_2(V)$. Therefore the variety
$\Gr_2(V)$ is $K$-spherical if and only if it is $K'$-spherical.

If $\Gr_2(V)$ is a spherical $K'$-variety then $V$ is a spherical
($K' \times \FF^\times$)-module by
Corollary~\ref{crl_V_is_spherical}.
Theorem~\ref{thm_spherical_modules_simple} implies that, up to a
geometrical equivalence, the pair $(K', V)$ is contained in
Table~\ref{table_spherical_modules}.

As was already mentioned in Introduction, every flag variety of the
group $\SL_n$ is spherical. Thus for every $n \ge 4$ the variety
$\Gr_2(\FF^n)$ is $\SL_n$-spherical.

The variety $\Gr_2(\FF^{2n})$ is $\Sp_{2n}$-spherical for $n \ge 2$
by Proposition~\ref{prop_sympl_grassm}.

The variety $\Gr_2(\FF^n)$ is $\SO_n$-spherical for $n \ge 4$ by
Proposition~\ref{prop_orth_grassm}.

Let us show that the variety $\Gr_2(\FF^8)$ is $\Spin_7$-spherical.
It is known (see \cite[Table~6, row~4]{Ela72} or \cite[\S\,5,
Proposition~26]{SK}) that under the natural action of the group
$\Spin_7 \times \GL_2$ on $\FF^8 \otimes \FF^2$ there is an open
orbit and the Lie algebra of the stabilizer of any point of this
orbit is isomorphic to $\mathfrak{gl}_3$. Applying
Corollary~\ref{crl_grassmannians}, we find that under the action of
the group $\Spin_7$ on $\Gr_2(\FF^8)$ there is an open orbit~$O$ and
the Lie algebra of the stabilizer of any point $x \in O$ is still
isomorphic to $\mathfrak{gl}_3$. It follows that the connected
component of the identity of the stabilizer of any point $x \in O$
is isomorphic to~$\GL_3$. It is well known (see, for
instance,~\cite[Table~1]{Kr79}) that $\GL_3$ is a spherical subgroup
of~$\Spin_7$.

We now prove that the variety $\Gr_2(V)$ is not $K'$-spherical for
the pairs $(K', V)$ in rows 4--9, 11--14 of
Table~\ref{table_spherical_modules}.

Applying Proposition~\ref{prop_inequality}, we obtain the following
necessary condition for $\Gr_2(V)$ to be $K'$-spherical:
\begin{equation} \label{eqn_inequality}
\dim K' + \rk K' \ge 4d - 8.
\end{equation}
A direct check shows that inequality~(\ref{eqn_inequality}) does not
hold for the pairs $(K', V)$ in rows 4, 5, 11--14 of
Table~\ref{table_spherical_modules}.

\begin{lemma} \label{lemma_K_V1otimesV2}
Suppose that $L$ is a connected reductive subgroup of\, $\SL_n
\times \SL_m$, where $n \ge m \ge 2$. Then the variety $\Gr_2(\FF^n
\otimes \FF^m)$ is $L$-spherical if and only if $n = m = 2$ and $L =
\SL_2 \times \SL_2$.
\end{lemma}

\begin{proof}
If $n = m = 2$ and $L = \SL_2 \times \SL_2$, then the pair $(L,
\FF^2 \otimes \nobreak \FF^2)$ is geometrically equivalent to the
pair $(\SO_4, \FF^4)$. By Proposition~\ref{prop_orth_grassm} the
action of $\SO_4$ on $\Gr_2(\FF^4)$ is spherical.

We now prove that the $L$-sphericity of the variety $\Gr_2(\FF^n
\otimes \FF^m)$ implies $n = m = 2$. Obviously, if $\Gr_2(\FF^n
\otimes\nobreak \FF^m)$ is $L$-spherical then it is also ($\SL_n
\times \SL_m$)-spherical. Therefore it suffices to prove the
required assertion in the case $L = \SL_n \times \SL_m$.

We divide our subsequent consideration into two cases.

\textit{Case}~1. $n > 2m$. We show that in this case the variety
$\Gr_2(\FF^n \otimes\nobreak \FF^m)$ is not ($\SL_n \times
\SL_m$)-spherical. For $k = n,m$ let $T_k$ denote the group of all
upper-triangular matrices in~$\GL_k$ and put $B_k = T_k \cap \SL_k$,
so that $B_k$ is a Borel subgroup of $\SL_k$. In view of
Proposition~\ref{prop_open_orbit_in_Fl} it suffices to prove that
the space $W = \FF^n \otimes \FF^m \otimes \FF^2$ contains no open
orbit for the action of $B_n \times B_m \times \GL_2$ or,
equivalently, for the action of $T_n \times T_m \times \SL_2$. In
turn, the latter will hold if we prove that the codimension of an
orbit of general position for the action of $T_n \times T_m$ on $W$
is at least~$4$.

We represent $W$ as the set of pairs of $(n \times m)$-matrices.
Then the action $T_n \times T_m : W$ is described by the formula
$((T, U), (P,Q)) \mapsto (TPU^\top, TQU^\top)$, where $(T,U) \in T_n
\times T_m$, $(P,Q) \in W$. It is not hard to check that, acting by
the group $T_n \times T_m$, one can reduce a pair $(P,Q)$ from a
suitable open subset of~$W$ to a uniquely determined canonical form
$(P', Q')$, where
$$
\arraycolsep=3pt
P' = \left(%
\begin{array}{ccc}%
 & 0 & \\ \hline
 & &  \\
 & 0 & \\
 & &
\\ \hline%
1 & & 0 \\
 & \dd & \\
0 & & 1
\end{array}
\right),%
\qquad
Q' = \left(%
\begin{array}{ccc}%
 & 0 & \\ \hline%
1 & & 0 \\
\vd & \dd & \\
1 & * & 1 \\\hline%
 & &  \\
 & * & \\
 & &
\end{array}
\right).
$$
(Each of the matrices $P',Q'$ is divided into three blocks: the
upper one consists of the upper $n-2m$ rows, the middle one consists
of the following $m$ rows, and the lower one consists of the last
$m$ rows.) It follows that the codimension of a $(T_n \times
T_m)$-orbit of general position in~$W$ equals $m^2 + (m-1)(m-2)/2
\ge m^2 \ge 4$ as required.

\textit{Case}~2. $m \le n \le 2m$. Suppose that $n = 2m - l$, where
$0 \le l \le m$. If the variety $\Gr_2(\FF^n \otimes \FF^m)$ is
$(\SL_n \times \SL_m)$-spherical, then
applying~(\ref{eqn_inequality}) we get the inequality
$$
(2m - l)^2 - 1 + m^2 - 1 + (2m - l) - 1 + m - 1 \ge 4(2m - l)m - 8,
$$
which takes the form
\begin{equation} \label{eqn_inequality2}
3m^2 - 3m - 4 \le l^2 - l
\end{equation}
after transformations. Since $0 \le l \le m$,
inequality~(\ref{eqn_inequality2}) implies that $3m^2 - 3m - 4 \le
m^2$. Hence $2m^2 - 3m - 4 \le 0$ and so $m = 2$. Then $l = 0,1
\text{ or } 2$, but the first two cases do not occur
by~(\ref{eqn_inequality2}). Thus $l = 2$, that is, $n = 2$.

It remains to show that the action $L : \Gr_2(\FF^2 \otimes \FF^2)$
is not spherical for every proper reductive subgroup $L \subset
\SL_2 \times \SL_2$. Indeed, in this case we have $\rk L \le 2$ and
$\dim L \le 5$, hence inequality (\ref{eqn_inequality}) does not
hold.

The proof of the lemma is completed.
\end{proof}

Lemma~\ref{lemma_K_V1otimesV2} immediately implies that the variety
$\Gr_2(V)$ is not $K'$-spherical for the pairs $(K', V)$ in rows
6--9 of Table~\ref{table_spherical_modules}, which completes the
proof of Theorem~\ref{thm_gr2_case_of_simple_V}.
\end{proof}

\subsection{Spherical actions on~$\Gr_2(V)$ in the case where
$V$ is a nonsimple $K$-module} \label{subsec_Gr2_non-simple_V}

In this subsection we suppose that $r \ge 2$. Here the main result
is Theorem~\ref{thm_gr2_case_of_nonsimple_V}.

\begin{proposition} \label{prop_either_sl_or_sp}
Suppose that $\Gr_2(V)$ is a spherical $K$-variety. Then for every
$i = 1, \ldots, r$ the pair $(K', V_i)$ is geometrically equivalent
to either of the pairs $(\SL_n, \FF^n)$ $(n \ge 1)$ or $(\Sp_{2n},
\FF^{2n})$ $(n \ge 2)$.
\end{proposition}

\begin{proof}
It follows from Proposition~\ref{prop_gr2_necessary} that $V_i
\otimes \FF^2$ is a spherical $(K \times \GL_2)$-module for every $i
= 1, \ldots, r$. The proof is completed by applying
Theorem~\ref{thm_spherical_modules_simple}.
\end{proof}

\begin{proposition} \label{prop_no_diagonal_action}
Suppose that $\Gr_2(V)$ is a spherical $K$-variety. Then every
simple normal subgroup in~$K'$ acts nontrivially on at most one
summand of decomposition~\textup{(\ref{eqn_decomposition_of_V})}.
\end{proposition}

\begin{proof}
Assume that there is a simple normal subgroup $K_0 \subset K'$
acting nontrivially on two different summands of
decomposition~(\ref{eqn_decomposition_of_V}). Without loss of
generality we shall assume that $K_0$ acts nontrivially on~$V_1$
and~$V_2$. (We note that $\dim V_1 \ge 2$ and $\dim V_2 \ge 2$ in
this case.) Proposition~\ref{prop_two_grassm_sphericity}(a) implies
that the variety $\Gr_2(V_1 \oplus V_2)$ is $K$-spherical. Making
use of Proposition~\ref{prop_either_sl_or_sp}, we find that, up to a
geometrical equivalence, the pair $(K', V_1 \oplus V_2)$ is
contained in Table~\ref{table_just_three}, where $K'$ is assumed to
act diagonally on~$V_1$ and~$V_2$ in all cases.

\begin{table}[!h]
\begin{center}
\caption{} \label{table_just_three}
\begin{tabular}{|c|c|c|c|c|}
\hline

No. & $K'$ & $V_1$ & $V_2$ & Note \\

\hline

1 & $\SL_n$ & $\FF^n$ & $\FF^n$ & $n \ge 2$ \\

\hline

2 & $\SL_n$ & $\FF^n$ & $(\FF^n)^*$ & $n \ge 3$ \\

\hline

3 & $\Sp_{2n}$ & $\FF^{2n}$ & $\FF^{2n}$ & $n \ge 2$ \\

\hline

\end{tabular}
\end{center}
\end{table}

For each case in Table~\ref{table_just_three} we put $Z = \Gr_2(V_1
\oplus V_2)$ and $X = \Gr_2(V_1)$. Let us show that the variety $Z$
is not $K$-spherical. We consider all the three cases separately.

\textit{Case}~1. If $n = 2$ then, by Corollary~\ref{crl_k=dimV1},
$Z$ is a $K$-spherical variety if and only if $V_1^* \otimes V_2$ is
a spherical $(C \times K')$-module, where $K$ acts diagonally
and~$C$ acts via the character $\chi_2 - \chi_1$. It follows that
$\FF^2 \otimes \FF^2$ is a spherical $(\SL_2 \times
\FF^\times)$-module, where $\SL_2$ acts diagonally and~$\FF^\times$
acts by scalar transformations. The latter is false since the
indicated module does not satisfy
inequality~(\ref{eqn_inequality_main}).

In what follows we suppose that $n \ge 3$. Applying
Proposition~\ref{prop_point+group_sl} to $X$ and then
Proposition~\ref{prop_two_grassm_sphericity}(b) to~$Z$ and~$X$, we
find a point $[W] \in X$ and a group $L \subset (K')_{[W]}$ with the
following properties:

(1) $L \simeq \mathrm{S}(\mathrm{L}_2 \times \mathrm{L}_{n-2})$;

(2) the pair $(L, W)$ is geometrically equivalent to the pair
$(\GL_2, \FF^2)$;

(3) the pair $(L, V_2)$ is geometrically equivalent to the pair
$(\mathrm{S}(\mathrm{L}_2 \times \mathrm{L}_{n-2}), \FF^2 \oplus
\FF^{n-2})$;

(4) the $K$-sphericity of $Z$ is equivalent to the sphericity of the
$(C \times L)$-module $W^* \otimes V_2$, where $L$ acts diagonally
and $C$ acts via the character $\chi_2 - \chi_1$.

In view of the $\SL_2$-module isomorphisms $(\FF^2)^* \simeq \FF^2$
and $\FF^2 \otimes \FF^2 \simeq \mathrm{S}^2 \FF^2 \oplus \FF^1$,
the sphericity of the $(C \times L)$-module $W^* \otimes V_2$
implies that the $(\SL_2 \times \SL_{n-2} \times
(\FF^\times)^2)$-module
$$
\mathrm S^2 \FF^2 \oplus (\FF^2 \otimes \FF^{n-2})
$$
is spherical, where $\SL_2$ acts diagonally on $\mathrm S^2 \FF^2$
and~$\FF^2$, $\SL_{n-2}$ acts on~$\FF^{n-2}$, and $(\FF^\times)^2$
acts on each of the direct summands by scalar transformations. By
Theorem~\ref{thm_spherical_modules_nonsimple} the indicated module
is not spherical.

\textit{Case}~2. Using an argument similar to that in Case~1 for $n
\ge 3$ we deduce from the condition of $Z$ being $K$-spherical that
the $(\SL_2 \times \SL_{n-2} \times (\FF^\times)^2)$-module
$$
\mathrm S^2 \FF^2 \oplus (\FF^2 \otimes (\FF^{n-2})^*)
$$
is spherical, where $\SL_2$ acts diagonally on $\mathrm S^2 \FF^2$
and~$\FF^2$, $\SL_{n-2}$ acts on $(\FF^{n-2})^*$, and
$(\FF^\times)^2$ acts on each of the direct summands by scalar
transformations. By Theorem~\ref{thm_spherical_modules_nonsimple}
the indicated module is not spherical.

\textit{Case}~3. If the variety $Z$ is $K$-spherical then $Z$ is
also $(C \times \SL_{2n})$-spherical (where $\SL_{2n}$ acts
diagonally on $V_1 \oplus V_2$). As was shown in Case~1, the latter
is false.
\end{proof}

\begin{theorem} \label{thm_gr2_case_of_nonsimple_V}
Suppose that $d \ge 4$ and $r \ge 2$. Then the variety $\Gr_2(V)$ is
$K$-spherical if and only if the following conditions hold:

\textup{(1)} up to a geometrical equivalence, the pair $(K', V)$ is
contained in Table~\textup{\ref{table_Gr2}};

\textup{(2)} the group $C$ satisfies the conditions listed in the
fourth column of Table~\textup{\ref{table_Gr2}}.
\end{theorem}

\begin{table}[!h]

\begin{center}

\caption{} \label{table_Gr2}

\begin{tabular}{|c|c|c|c|c|}
\hline

No. & $K'$ & $V$ & Conditions on $C$ & Note \\

\hline

1 & $\SL_n \times \SL_m$ & $\FF^n \oplus \FF^m$ & $\chi_1 \ne
\chi_2$ for $n = m = 2$ &
\begin{tabular}{c}
$n \ge m \ge 1$, \\
$n + m \ge 4$
\end{tabular}
\\

\hline

2 & $\Sp_{2n} \times \SL_m$ & $\FF^{2n} \oplus \FF^m$ &
$\chi_1 \ne \chi_2$ for $m = 2$ & $n \ge 2$, $m \ge 1$\\

\hline

3 & $\Sp_{2n} \times \Sp_{2m}$ & $\FF^{2n} \oplus \FF^{2m}$ &
$\chi_1 \ne \chi_2$ & $n \ge m \ge 2$ \\

\hline

4 & $\SL_n \times \SL_m \times \SL_l$ & $\FF^n \oplus \FF^m \oplus
\FF^l$ &

\begin{tabular}{c}
$\chi_2 - \chi_1, \chi_3 - \chi_1$ \\ lin. ind. for $n = 2$; \\

\hline

$\chi_2 \ne \chi_3$ for \\
$n \ge 3$, $m \le 2$ \\
\end{tabular}

&
\begin{tabular}{c}
$n \ge m \ge l \ge 1$, \\
$n \ge 2$
\end{tabular}
\\

\hline

5 & $\Sp_{2n} \times \SL_m \times \SL_l$ & $\FF^{2n} \oplus \FF^m
\oplus \FF^l$ &

\begin{tabular}{c}
$\chi_2 - \chi_1, \chi_3 - \chi_1$ \\ lin. ind. for $m \le 2$ \\

\hline

$\chi_1 \ne \chi_3$ for \\ $m \ge 3, l \le 2$ \\
\end{tabular}

&

\begin{tabular}{c}
$n \ge 2$, \\ $m \ge l \ge 1$
\end{tabular}
\\

\hline

6 & $\Sp_{2n} \times \Sp_{2m} \times \SL_l$ & $\FF^{2n} \oplus
\FF^{2m} \oplus \FF^l$ &

\begin{tabular}{c}
$\chi_2 - \chi_1, \chi_3 - \chi_1$ \\ lin. ind. for $l \le 2$ \\

\hline

$\chi_1 \ne \chi_2$ for $l \ge 3$ \\
\end{tabular}

&
\begin{tabular}{c}
$n \ge m \ge 2$, \\
$l \ge 1$
\end{tabular}
\\

\hline

7 & $\Sp_{2n} \times \Sp_{2m} \times \Sp_{2l}$ & $\FF^{2n} \oplus
\FF^{2m} \oplus \FF^{2l}$ & \begin{tabular}{c} $\chi_1 - \chi_2,
\chi_1 - \chi_3$ \\ lin. ind. \end{tabular} & $n \ge m \ge l \ge 2$ \\

\hline

\end{tabular}

\end{center}

\end{table}

\begin{table}[!h]

\begin{center}

\caption{} \label{table_Gr2_proof}

\begin{tabular}{|c|c|c|}
\hline

Case & References & $(M, W)$ \\

\hline

\begin{tabular}{c}1, \\ $n {=} m {=} 2$ \end{tabular}
& \ref{crl_k=dimV1} & $(\SL_2, \FF^2)$ \\

\hline

1, $n \ge 3$ & \ref{prop_point+group_sl},
\ref{prop_two_grassm_sphericity}(b) & $(\GL_2, \FF^2)$ \\

\hline

2 & \ref{prop_point+group_sp_even},
\ref{prop_two_grassm_sphericity}(b)
& $(\SL_2, \FF^2)$ \\

\hline

3 & \ref{prop_point+group_sp_even},
\ref{prop_two_grassm_sphericity}(b)
& $(\SL_2, \FF^2)$ \\

\hline

4, $n = 2$ & \ref{crl_k=dimV1} & $(\SL_2, \FF^2)$ \\

\hline

4, $n \ge 3$ & \ref{prop_point+group_sl},
\ref{prop_two_grassm_sphericity}(b) & $(\GL_2, \FF^2)$ \\

\hline

5 & \ref{prop_point+group_sp_even},
\ref{prop_two_grassm_sphericity}(b)
& $(\SL_2, \FF^2)$ \\

\hline

6 & \ref{prop_point+group_sp_even},
\ref{prop_two_grassm_sphericity}(b)
& $(\SL_2, \FF^2)$ \\

\hline

7 & \ref{prop_point+group_sp_even},
\ref{prop_two_grassm_sphericity}(b)
& $(\SL_2, \FF^2)$ \\

\hline
\end{tabular}

\end{center}

\end{table}

\begin{proof}
Put $U = V_2 \oplus \ldots \oplus V_r$. Let $K_1$ (resp. $K_2$) be
the image of~$K'$ in~$\GL(V_1)$ (resp. $\GL(U)$).

If the variety $\Gr_2(V)$ is spherical with respect to the action
of~$K$, then it is also spherical with respect to the action of
$\GL(V_1) \times \ldots \times \GL(V_r)$. Then it follows from
Theorem~\ref{thm_MWZ} that $r \le 3$. Applying
Propositions~\ref{prop_either_sl_or_sp}
and~\ref{prop_no_diagonal_action} we find that, up to a geometrical
equivalence, the pair $(K', V)$ is contained in
Table~\ref{table_Gr2}. The subsequent reasoning is similar for each
of the cases in Table~\ref{table_Gr2}; the key points of the
arguments are gathered in Table~\ref{table_Gr2_proof}. First,
applying an appropriate combination of
statements~\ref{prop_point+group_sl},
\ref{prop_point+group_sp_even}, \ref{prop_two_grassm_sphericity}(b),
and~\ref{crl_k=dimV1} (see the column ``References'') to the
varieties $Z = \Gr_2(V)$ and $X = \Gr_2(V_1)$, we find a connected
reductive subgroup $L \subset K$ and an $L$-module $R$ with the
following property: $Z$~is a spherical $K$-variety if and only if
$R$ is a spherical $L$-module. After that the sphericity of the
$L$-module $R$ is verified using
Theorems~\ref{thm_spherical_modules_simple}
and~\ref{thm_spherical_modules_nonsimple}. Since the group $C$ acts
trivially on~$X$, we have $C \subset L$. Therefore to describe the
action $L : R$ it suffices to describe the actions $(L \cap K') : R$
and $C : R$. In all the cases we have $L \cap K' = M \times K_2$ for
some subgroup $M \subset K_1$. Moreover, $R = W^* \otimes U$, where
$M$ acts on~$W^*$ and $K_2$ acts on~$U$. For each of the cases, up
to a geometrical equivalence, the pair $(M, W)$ is indicated in the
third column of Table~\ref{table_Gr2_proof}. The action of $C$ on
$W$ is the same as on~$V_1$ and the action of $C$ on $U$ coincides
with the initial one.
\end{proof}

\subsection{Spherical actions on Grassmannians}
\label{subsec_grassmannians}

In this subsection we complete the description of spherical actions
on Grassmannians initiated
in~\S\S\,\ref{subsec_Gr2_simple_V},~\ref{subsec_Gr2_non-simple_V}.
The main result of this subsection is the following theorem.

\begin{theorem} \label{thm_grassmannians}
Suppose that $d \ge 6$ and $3 \le k \le d / 2$. Then the variety $X
= \Gr_k(V)$ is $K$-spherical if and only if the following conditions
hold:

\textup{(1)} up to a geometrical equivalence, the pair $(K', V)$ is
contained in Table~\textup{\ref{table_Gr_all}};

\textup{(2)} the number $k$ satisfies the conditions listed in the
fourth column of Table~\textup{\ref{table_Gr_all}};

\textup{(3)} the group~$C$ satisfies the conditions listed in the
fifth column of Table~\textup{\ref{table_Gr_all}}.
\end{theorem}

\begin{table}[h]

\begin{center}

\caption{} \label{table_Gr_all}

\begin{tabular}{|c|c|c|c|c|c|}
\hline

No. & $K'$ & $V$ & \begin{tabular}{c} Conditions on~$k$
\end{tabular}
& Conditions on $C$ & Note \\

\hline

1 & $\SL_n$ & $\FF^n$ & &
& $n \ge 6$ \\

\hline

2 & $\Sp_{2n}$ & $\FF^{2n}$ & & & $n \ge 3$ \\

\hline

3 & $\SO_n$ & $\FF^n$ & & & $n \ge 6$ \\

\hline

4 & $\SL_n \times \SL_m$ & $\FF^n \oplus \FF^m$ & & $\chi_1 \ne
\chi_2$ for $n = m = k$

&
\begin{tabular}{c}
$n \ge m$, \\ $n + m \ge 6$
\end{tabular}
\\

\hline

5 & $\Sp_{2n}$ & $\FF^{2n} \oplus \FF^1$ & & & $n \ge 3$ \\

\hline

6 & $\Sp_{2n} \times \SL_m$ & $\FF^{2n} \oplus \FF^m$ & $k = 3$ &
$\chi_1 \ne \chi_2$ for $m \le 3$

& $n, m \ge 2$ \\

\hline

7 & $\Sp_4 \times \SL_m$ & $\FF^4 \oplus \FF^m$ & $k \ge 4$ &
$\chi_1 \ne \chi_2$ for $k = m = 4$

& $m \ge 4$ \\

\hline

8 & $\SL_n \times \SL_m$ & $\FF^n \oplus \FF^m \oplus \FF^1$ & &
\begin{tabular}{c}
$\chi_2 - \chi_1, \chi_3 - \chi_1$ \\ lin. ind. for $k = n$;\\
\hline $\chi_2 \ne \chi_3$ for $m \le k < n$
\end{tabular} &

\begin{tabular}{c}
$n \ge m \ge 1$, \\ $n + m \ge 5$
\end{tabular}
\\

\hline

\end{tabular}

\end{center}

\end{table}

In the fourth column of Table~\ref{table_Gr_all} the empty cells
mean that $k$ may be any number such that $3 \le k \le d /2$.

\begin{proof}[Proof of Theorem~\textup{\ref{thm_grassmannians}}]
We first consider the case $r = 1$, that is, the case where $V$ is a
simple $K$-module. In this situation the center of $K$ acts
trivially on $\Gr_k(V)$, hence $\Gr_k(V)$ is $K$-spherical if and
only if it is $K'$-spherical.

If $\Gr_k(V)$ is $K'$-spherical, then by
Proposition~\ref{prop_Gr2_is_spherical} and
Theorem~\ref{thm_gr2_case_of_simple_V} the pair $(K', V)$ is
geometrically equivalent to one of the pairs in
Table~\ref{table_Gr2_simple_V}.

Since every flag variety of the group $\SL_n$ is $\SL_n$-spherical,
then so is $\Gr_k(\FF^n)$.

For $n \ge 3$ and $3 \le k \le n$ the variety $\Gr_k(\FF^{2n})$ is
$\Sp_{2n}$-spherical by Proposition~\ref{prop_sympl_grassm}.

For $n \ge 6$ and $3 \le k \le n/2$ the variety $\Gr_k(\FF^n)$ is
$\SO_n$-spherical by Proposition~\ref{prop_orth_grassm}.

For $3 \le k \le 4$ the variety $\Gr_k(\FF^8)$ is not
$\Spin_7$-spherical since inequality~(\ref{eqn_inequality_main})
does not hold in this case.

We now consider the case $r \ge 2$. By
Proposition~\ref{prop_Gr2_is_spherical} the $K$-sphericity of
$\Gr_k(V)$ implies the $K$-sphericity of $\Gr_2(V)$. Applying
Theorems~\ref{thm_gr2_case_of_nonsimple_V} and~\ref{thm_MWZ} we find
that the pair $(K', V)$ is geometrically equivalent to one of the
pairs in Table~\ref{table_Gr_all_aux}.

\begin{table}[h]

\begin{center}

\caption{} \label{table_Gr_all_aux}

\begin{tabular}{|c|c|c|c|}
\hline

No. & $K'$ & $V$ & Note \\

\hline

1 & $\SL_n \times \SL_m$ & $\FF^n \oplus \FF^m$ &
$n \ge m \ge 1$, $n + m \ge 6$ \\

\hline

2 & $\Sp_{2n} \times \SL_m$ & $\FF^{2n} \oplus \FF^m$ &
$n \ge 2$, $m \ge 1$, $2n + m \ge 6$\\

\hline

3 & $\Sp_{2n} \times \Sp_{2m}$ & $\FF^{2n} \oplus \FF^{2m}$ &
$n \ge m \ge 2$ \\

\hline

4 & $\SL_n \times \SL_m$ & $\FF^n \oplus \FF^m \oplus \FF^1$ & $n
\ge
m \ge 1$, $n + m \ge 5$ \\

\hline

5 & $\Sp_{2n} \times \SL_m$ & $\FF^{2n} \oplus \FF^m \oplus \FF^1$ &
$n \ge 2$, $m \ge 1$
\\

\hline

6 & $\Sp_{2n} \times \Sp_{2m}$ & $\FF^{2n} \oplus \FF^{2m} \oplus
\FF^1$ & $n \ge m \ge 2$ \\

\hline

\end{tabular}

\end{center}

\end{table}

\begin{table}[h!]

\begin{center}

\caption{} \label{table_Gr_all_proof}

\begin{tabular}{|c|c|c|c|}
\hline

No. & $(K_1, V_1)$ & Case & ($M, W$) \\

\hline

1 &
\begin{tabular}{c}
$(\SL_n, \FF^n)$,\\  $k = n$
\end{tabular}
& \ref{crl_k=dimV1} & $(\SL_n, \FF^n)$ \\

\hline

2 &
\begin{tabular}{c} $(\SL_n, \FF^n)$, \\ $k < n$
\end{tabular}
& \ref{prop_point+group_sl}, \ref{prop_two_grassm_sphericity}(b)
& $(\GL_k, \FF^k)$ \\

\hline

3 &
\begin{tabular}{c} $(\Sp_{2n}, \FF^{2n})$, \\ $k \le n, k = 2l$
\end{tabular}
& \ref{prop_point+group_sp_even},
\ref{prop_two_grassm_sphericity}(b) & $(\underbrace{\SL_2 {\times}
\ldots {\times} \SL_2}_l$, $\underbrace{\FF^2 {\oplus} \ldots
{\oplus} \FF^2}_l)$
\\

\hline

4 &
\begin{tabular}{c}
$(\Sp_{2n}, \FF^{2n})$, \\ $k \le n$, $k = 2l + 1$
\end{tabular}
& \ref{prop_point+group_sp_odd}, \ref{prop_two_grassm_sphericity}(b)
& $(\FF^\times {\times} \underbrace{\SL_2 {\times} \ldots {\times}
\SL_2}_l,
\FF^1 {\oplus} \underbrace{\FF^2 {\oplus} \ldots {\oplus} \FF^2}_l)$ \\

\hline

5 &
\begin{tabular}{c}
$(\Sp_{2n}, \FF^{2n})$, \\ $n < k < 2n$, \\ $2n - k = 2l$
\end{tabular}
& \ref{crl_point+group_sp_even}, \ref{prop_two_grassm_sphericity}(b)
&
\begin{tabular}{c}
$(\underbrace{\SL_2 {\times} \ldots {\times} \SL_2}_l {\times}
\Sp_{2n-4l}$, \\ $\underbrace{\FF^2 {\oplus} \ldots {\oplus}
\FF^2}_l {\oplus} \FF^{2n-4l})$
\end{tabular}
\\

\hline

6 &
\begin{tabular}{c}
($\Sp_{2n}, \FF^{2n}$), \\ $n < k < 2n$, \\ $2n - k = 2l + 1$
\end{tabular}
& \ref{crl_point+group_sp_odd}, \ref{prop_two_grassm_sphericity}(b)
&
\begin{tabular}{c}
$(\FF^\times {\times} \underbrace{\SL_2 {\times} \ldots {\times}
\SL_2}_l {\times} \Sp_{2n-4l-2}$,\\
$\FF^1 {\oplus} \underbrace{\FF^2 {\oplus} \ldots {\oplus} \FF^2}_l
{\oplus} \FF^{2n-4l-2})$
\end{tabular}
\\

\hline

7 &
\begin{tabular}{c}
$(\Sp_{2n}, \FF^{2n})$, \\ $k = 2n$
\end{tabular}
& \ref{crl_k=dimV1} & $(\Sp_{2n}, \FF^{2n})$ \\

\hline

\end{tabular}

\end{center}

\end{table}

For each case in Table~\ref{table_Gr_all_aux} we denote by $K_1$
(resp.~$K_2$) the first (resp. second) factor of~$K'$. We also put
$$
U = V_2 \oplus \ldots \oplus V_r.
$$
Up to changing the order of factors of~$K$ (along with
simultaneously interchanging the first and second summands of~$V$),
the pair $(K_1, V_1)$ fits in at least one of the cases listed in
the second column of Table~\ref{table_Gr_all_proof}. In all these
cases the subsequent reasoning is similar. At first, applying an
appropriate combination of statements
\ref{prop_point+group_sl}--\ref{prop_two_grassm_sphericity},
\ref{crl_k=dimV1} (see the third column of
Table~\ref{table_Gr_all_proof}) to $Z = \Gr_k(V)$ and $X =
\Gr_k(V_1)$, we find a connected reductive subgroup $L \subset K$
and an $L$-module $R$ with the following property: $Z$~is a
spherical $K$-variety if and only if $R$ is a spherical $L$-module.
After that the sphericity of the $L$-module $R$ is verified using
Theorems~\ref{thm_spherical_modules_simple},
\ref{thm_spherical_modules_nonsimple},
and~\ref{thm_spherical_modules_general}. Since $C$ acts trivially
on~$X$, we have $C \subset L$. Therefore to describe the action $L :
R$ it suffices to describe the actions $(L \cap K') : R$ and $C :
R$. In all the cases we have $L \cap K' = M \times K_2$ for some
subgroup $M \subset K_1$. Moreover, $R = W^* \otimes U$ where $M$
acts on~$W^*$ and $K_2$ acts on~$U$. For each of the cases, up to a
geometrical equivalence, the pair $(M, W)$ is indicated in the
fourth column of Table~\ref{table_Gr_all_proof}. The action of $C$
on~$W$ is the same as on~$V_1$ and the action of $C$ on~$U$
coincides with the initial one.
\end{proof}

\subsection{Completion of the classification}
\label{subsec_finish}

In this subsection we classify all spherical actions on $V$-flag
varieties that are not Grassmannians (see
Theorem~\ref{thm_non-grassmannians}). Thereby we complete the proof
of Theorem~\ref{thm_main_theorem}.

Let $\mathbf a = (a_1, \ldots, a_s)$ be a composition of~$d$ such
that $a_1 \le \ldots \le a_s$ and $s \ge 3$. (The latter exactly
means that $\Fl_{\mathbf a}(V)$ is not a Grassmannian.)

\begin{theorem} \label{thm_non-grassmannians}
The variety $\Fl_{\mathbf a}(V)$ is $K$-spherical if and only if the
following conditions hold:

\textup{(1)} the pair $(K', V)$, considered up to a geometrical
equivalence, and the tuple $(a_1, \ldots, a_{s-1})$ are contained in
Table~\textup{\ref{table_non-grassmannians}};

\textup{(2)} the group $C$ satisfies the conditions listed in the
fifth column of Table~\textup{\ref{table_non-grassmannians}}.
\end{theorem}

\begin{table}[h]

\begin{center}

\caption{} \label{table_non-grassmannians}

\begin{tabular}{|c|c|c|c|c|c|}
\hline

No. & $K'$ & $V$ & $(a_1, \ldots, a_{s-1})$ & Conditions on $C$
& Note \\

\hline

1 & $\SL_n$ & $\FF^n$ & $(a_1, \ldots, a_{s-1})$ & & $n \ge 3$ \\

\hline

2 & $\Sp_{2n}$ & $\FF^{2n}$ & $(1, a_2)$ & & $n \ge 2$ \\

\hline

3 & $\Sp_{2n}$ & $\FF^{2n}$ & $(1,1,1)$ & & $n \ge 2$ \\

\hline

4 & $\SL_n$ & $\FF^n \oplus \FF^1$ & $(a_1, \ldots, a_{s-1})$ &
$\chi_1 \ne \chi_2$ for $s = n+1$ & $n \ge 2$ \\

\hline

5 & $\SL_n \times \SL_m$ & $\FF^n \oplus \FF^m$ & $(1, a_2)$ &
$\chi_1 \ne \chi_2$ for $n = 1 + a_2$ & $n \ge m \ge 2$ \\

\hline

6 & $\SL_n \times \SL_2$ & $\FF^n \oplus \FF^2$ & $(a_1, a_2)$
&
\begin{tabular}{c}
$\chi_1 \ne \chi_2$ for \\
$n = 4$ and $a_1 = a_2 = 2$
\end{tabular}
&
\begin{tabular}{c}
$n \ge 4$,\\ $a_1 \ge 2$
\end{tabular}
\\

\hline

7 & $\Sp_{2n} \times \SL_m$ & $\FF^{2n} \oplus \FF^m$ & $(1,1)$ &
$\chi_1 \ne \chi_2$ for $m \le 2$ &
\begin{tabular}{c}
$n \ge 2$, \\ $m \ge 1$
\end{tabular}
\\

\hline

8 & $\Sp_{2n} \times \Sp_{2m}$ & $\FF^{2n} \oplus \FF^{2m}$ &
$(1,1)$ & $\chi_1 \ne \chi_2$ & $n \ge m \ge 2$\\

\hline

\end{tabular}

\end{center}

\end{table}

In the proof of Theorem~\ref{thm_non-grassmannians} we shall need
several auxiliary results.

\begin{proposition} \label{prop_Fl(1,k;V)}
Suppose that $n \ge 3$, $V = \FF^{2n}$, $K = \Sp_{2n}$, and $2 \le k
\le n-1$. Then the variety $\Fl(1,k; V)$ is $K$-spherical.
\end{proposition}

\begin{proof}
Put $Z = \Fl(1,k; V)$, $X = \PP(V)$ and consider the natural
$K$-equivariant morphism $\varphi \colon Z \to X$. Using
Proposition~\ref{prop_point+group_sp_odd} and then
Definition~\ref{dfn_P-property}, we find that there are a point $[W]
\in X$ and a connected reductive subgroup $L \subset K_{[W]}$ with
the following properties:

(1) the pair $(L, V/W)$ is geometrically equivalent to the pair
$(\FF^\times \times \Sp_{2n-2}, \FF^1 \oplus \FF^{2n-2})$;

(2) the $K$-sphericity of $Z$ is equivalent to the $L$-sphericity of
$\varphi^{-1}([W]) \simeq \Gr_k(V/W)$.

By Theorems~\ref{thm_gr2_case_of_nonsimple_V}
and~\ref{thm_grassmannians} the variety $\Gr_k(V/W)$ is
$L$-spherical, which completes the proof.
\end{proof}

\begin{proposition} \label{prop_Fl(2,2;V)}
Suppose that $n \ge 3$, $V = \FF^{2n}$, and $K = \Sp_{2n}$. Then the
variety $\Fl(2,2; V)$ is not $K$-spherical.
\end{proposition}

\begin{proof}
Put $Z = \Fl(2,2; V)$, $X = \Gr_2(V)$ and consider the natural
$K$-equivariant morphism $\varphi \colon Z \to X$. Applying
Proposition~\ref{prop_point+group_sp_even} and taking into account
Definition~\ref{dfn_P-property}, we find that there are a point $[W]
\in X$ and a connected reductive subgroup $L \subset K_{[W]}$ with
the following properties:

(1) the pair $(L, V/W)$ is geometrically equivalent to the pair
$$
(\SL_2 \times \Sp_{2n-4}, \FF^2 \oplus \FF^{2n-4});
$$

(2) the $K$-sphericity of $Z$ is equivalent to the $L$-sphericity of
$\varphi^{-1}([W]) \simeq \Gr_2(V/W)$.

By Theorem~\ref{thm_gr2_case_of_nonsimple_V} the variety
$\Gr_2(V/W)$ is not $L$-spherical, which completes the proof.
\end{proof}

\begin{proposition} \label{prop_Fl(1,2;V)}
Suppose that $n \ge 2$, $m \ge 1$, $V_1 = \FF^{2n}$, $V_2 = \FF^m$,
$V = V_1 \oplus V_2$, $K_1 = \Sp_{2n}$, $K_2 = \GL_m$, and $K = K_1
\times K_2$. Then the variety $\Fl(1,2; V)$ is not $K$-spherical.
\end{proposition}

\begin{proof}
Put $Z = \Fl(1,2; V)$ and $X = \Gr_3(V_1)$. Applying
Proposition~\ref{prop_point+group_sp_odd} (for $n \ge 3$) or
Corollary~\ref{crl_point+group_sp_odd} (for $n = 2$) to~$X$ and then
Proposition~\ref{prop_Fl(r1,r2;V)} to $Z$ and~$X$, we find that
there are a point $[W] \in X$ and a connected reductive subgroup $L
\subset K_{[W]}$ with the following properties:

(1) $L = L_1 \times K_2$, where $L_1 \subset K_1$;

(2) $L_1 \simeq \FF^\times \times \SL_2$ for $2 \le n \le 3$ and
$L_1 \simeq \FF^\times \times \SL_2 \times \Sp_{2n-6}$ for $n \ge
4$;

(3) the pair $(L_1, W)$ is geometrically equivalent to the pair
$(\FF^\times \times \SL_2, \FF^1 \oplus \FF^2)$;

(4) the $K$-sphericity of $Z$ is equivalent to the $L$-sphericity of
$(W^* \otimes V_2) \times \PP(W)$, where $L_1$ acts diagonally on
$W^*$ and~$\PP(W)$, and $K_2$ acts on~$V_2$.

It is easy to see that the $L$-sphericity of $(W^* \otimes V_2)
\times \PP(W)$ is equivalent to the sphericity of the $(L_1 \times
K_2 \times \FF^\times)$-module $(W^* \otimes V_2) \oplus W$, where
$L_1$ acts diagonally on $W^*$ and~$W$, $K_2$ acts on~$V_2$, and
$\FF^\times$ acts on the summand~$W$ by scalar transformations.
Applying Theorem~\ref{thm_spherical_modules_nonsimple} we find that
the indicated module is not spherical, which completes the proof.
\end{proof}

\begin{proof}[Proof of Theorem~\textup{\ref{thm_non-grassmannians}}]
Throughout this proof, the description of the partial order
$\preccurlyeq$ on the set $\mathscr F(\GL(V)) / \! \sim$ (see
Corollary~\ref{crl_partial_order} and
\S\,\ref{subsec_Young_diagrams}) as well as
Theorem~\ref{thm_prec_spherical} will be use without extra
explanation.

Since $\Fl_{\mathbf a}(V)$ is not a Grassmannian, we have
$$
\bl \Fl_{\mathbf a}(V) \br \succcurlyeq \bl \Fl(1,1; V) \br.
$$
Therefore the $K$-sphericity of $\Fl_{\mathbf a}(V)$ implies the
$K$-sphericity of $\Fl(1,1; V)$. The following proposition provides
a complete classification of pairs $(K, V)$ for which $K$ acts
spherically on $\Fl(1,1; V)$.

\begin{proposition} \label{prop_Fl(1,1;V)}
For $d \ge 3$, the variety $\Fl(1,1; V)$ is $K$-spherical if and
only if the following conditions hold:

\textup{(1)} up to a geometrical equivalence, the pair $(K', V)$ is
contained in Table~\textup{\ref{table_Fl(1,2;V)}};

\textup{(2)} the group $C$ satisfies the conditions listed in the
fourth column of Table~\textup{\ref{table_Fl(1,2;V)}}.
\end{proposition}

\begin{table}[h]

\begin{center}

\caption{} \label{table_Fl(1,2;V)}

\begin{tabular}{|c|c|c|c|c|}
\hline

No. & $K'$ & $V$ & Conditions on $C$ & Note \\

\hline

1 & $\SL_n$ & $\FF^n$ & & $n \ge 3$ \\

\hline

2 & $\Sp_{2n}$ & $\FF^{2n}$ & & $n \ge 2$ \\

\hline

3 & $\SL_n$ & $\FF^n \oplus \FF^1$ & $\chi_1 \ne \chi_2$
for $n = 2$ & $n \ge 2$ \\

\hline

4 & $\SL_n \times \SL_m$ & $\FF^n \oplus \FF^m$ &
$\chi_1 \ne \chi_2$ for $n = m = 2$ & $n \ge m \ge 2$ \\

\hline

5 & $\Sp_{2n} \times \SL_m$ & $\FF^{2n} \oplus \FF^m$ &
$\chi_1 \ne \chi_2$ for $m \le 2$ & $n \ge 2$, $m \ge 1$ \\

\hline

6 & $\Sp_{2n} \times \Sp_{2m}$ & $\FF^{2n} \oplus \FF^{2m}$
& $\chi_1 \ne \chi_2$ & $n,m \ge 2$\\

\hline

\end{tabular}

\end{center}

\end{table}

\begin{proof}
By Corollary~\ref{crl_Fl(1,1,...,1;V)} the action $K : \Fl(1,1; V)$
is spherical if and only if $V \otimes \FF^2$ is a spherical $(K
\times \GL_2)$-module. Now the required result follows from
Theorems~\ref{thm_spherical_modules_simple}
and~\ref{thm_spherical_modules_nonsimple}.
\end{proof}

In view of Proposition~\ref{prop_Fl(1,1;V)}, to complete the proof
of Theorem~\ref{thm_non-grassmannians} it remains to find all
$K$-spherical $V$-flag varieties~$X$ with~$\bl X \br \succ \bl
\Fl(1,1; V) \br$ for all the cases in Table~\ref{table_Fl(1,2;V)}.
In what follows we consider each of these cases separately.

\textit{Case}~1. $K' = \SL_n, V = \FF^n$. We have $\bl X \br
\preccurlyeq \bl \Fl_{\mathbf b}(V) \br$ where $\mathbf b = (1,
\ldots, 1)$. By Corollary~\ref{crl_Fl(1,1,...,1;V)} and Theorem
~\ref{thm_spherical_modules_simple} the variety $\Fl_{\mathbf b}(V)$
is $K$-spherical, hence so is~$X$.

\textit{Case}~2. $K' = \Sp_{2n}, V = \FF^{2n}$. If $s = 3$ and $a_1
= 1$, then $X$ is $K$-spherical by Proposition~\ref{prop_Fl(1,k;V)}.
If $s = 3$ and $a_1 \ge 2$, then $\bl X \br \succcurlyeq \bl
\Fl(2,2; V) \br$, and so $X$ is not $K$-spherical by
Proposition~\ref{prop_Fl(2,2;V)}. If $s = 4$ and $a_1 = a_2 = a_3 =
1$, then $X$ is $K$-spherical by Corollary~\ref{crl_Fl(1,1,...,1;V)}
and Theorem~\ref{thm_spherical_modules_simple}. If $s = 4$ and $a_3
\ge 2$, then $\bl X \br \succcurlyeq \bl \Fl(2,2; V) \br$, and so
$X$ is not $K$-spherical by Proposition~\ref{prop_Fl(2,2;V)}. In
what follows we assume that $s \ge 5$. Then $\bl X \br \succcurlyeq
\bl \Fl(1,1,1,1; V) \br$. The variety $\Fl(1,1,1,1; V)$ is not
$K$-spherical by Corollary~\ref{crl_Fl(1,1,...,1;V)} and
Theorem~\ref{thm_spherical_modules_simple}, hence $X$ is not
$K$-spherical either.

\textit{Case}~3. $K' = \SL_n, V = \FF^n \oplus \FF^1$. If $s = n +
1$, that is, $\mathbf a = (1, \ldots, 1)$, then by
Corollary~\ref{crl_Fl(1,1,...,1;V)} and
Theorem~\ref{thm_spherical_modules_nonsimple} the variety $X$ is
$K$-spherical if and only if $\chi_1 \ne\nobreak \chi_2$. In what
follows we assume that $s \le n$. Then $\bl X \br \preccurlyeq \bl
\Fl_{\mathbf b}(V) \br$, where $\mathbf b = (1, \ldots, 1, 2)$. By
Corollary~\ref{crl_Fl(1,1,...,1;V)} and
Theorem~\ref{thm_spherical_modules_nonsimple} the variety
$\Fl_{\mathbf b}(V)$ is $K'$-spherical, hence $X$ is $K$-spherical.

\textit{Case}~4. $K' = \SL_n \times \SL_m$, $V = \FF^n \oplus
\FF^m$, $n \ge m \ge 2$. Put $K_1 = \SL_n$, $K_2 = \SL_m$, $V_1 =
\FF^n$, $V_2 = \FF^m$. If $s \ge 4$ then $\bl X \br \succcurlyeq \bl
\Fl(1,1,1; V) \br$. By Corollary~\ref{crl_Fl(1,1,...,1;V)} and
Theorem~\ref{thm_spherical_modules_nonsimple} the variety
$\Fl(1,1,1; V)$ is not $K$-spherical, hence $X$ is not $K$-spherical
either. In what follows we assume that $s = 3$. Put $k = a_1 + a_2$.
We consider the following two subcases: $a_1 = 1$ and $a_1 \ge 2$.

\textit{Subcase}~4.1. $a_1 = 1$. Then $k \le n$. In view of
Propositions~\ref{prop_point+group_sl} and~\ref{prop_Fl(r1,r2;V)}
there are a point $[W] \in \Gr_k(V_1)$ and a subgroup $L \subset
(K_1)_{[W]}$ with the following properties:

(1) the pair $(L, W)$ is geometrically equivalent to the pair
$(\GL_k, \FF^k)$ for $k < n$ and the pair $(\SL_k, \FF^k)$ for $k =
n$;

(2) the $K$-sphericity of $X$ is equivalent to the $(C \times L
\times K_2)$-sphericity of the variety $(W^* \otimes V_2) \times
\PP(W)$, where $L$ acts diagonally on~$W^*$ and~$W$, $K_2$~acts on
$V_2$, and~$C$ acts on $W^*$, $V_2$, and $W$ via the characters
$-\chi_1$, $\chi_2$, and $\chi_1$, respectively.

It follows from~(2) that the $K$-sphericity of $X$ is equivalent to
the sphericity of the $(C \times L \times K_2 \times
\FF^\times)$-module $(W^* \otimes V_2) \oplus W$, where $C$, $L$,
and $K_2$ act as described in~(2) and $\FF^\times$ acts on the
summand~$W$ by scalar transformations. Applying
Theorem~\ref{thm_spherical_modules_nonsimple} we find that the
indicated module is not spherical when $k = n$, $\chi_1 = \chi_2$
and is spherical in all other cases.

\textit{Subcase}~4.2. $a_1 \ge 2$. By Theorem~\ref{thm_MWZ} the
$K$-sphericity of~$X$ implies that $m = 2$. Then $k \le n$ and the
equality is attained if and only if $a_1 = a_2 = 2$ and $n = 4$. In
view of Propositions~\ref{prop_point+group_sl}
and~\ref{prop_Fl(r1,r2;V)} there are a point $[W] \in \Gr_k(V_1)$
and a subgroup $L \subset (K_1)_{[W]}$ with the following
properties:

(1) the pair $(L, W)$ is geometrically equivalent to the pair
$(\GL_k, \FF^k)$ for $k < n$ and the pair $(\SL_k, \FF^k)$ for $k =
n$;

(2) the $K$-sphericity of $X$ is equivalent to the $(C \times L
\times K_2)$-sphericity of the variety $Y = (W^* \otimes V_2) \times
\Gr_{a_1}(W)$ be, where $L$ acts diagonally on~$W^*$ and
$\Gr_{a_1}(W)$, $K_2$~acts on~$V_2$, and $C$ acts on $W^*$ and $V_2$
via the characters $-\chi_1$ and $\chi_2$, respectively.

Applying Proposition~\ref{prop_point+group_sl} to $\Gr_{a_1}(W)$ and
then considering the natural projection $Y \to \Gr_{a_1}(W)$, we
find that there is a subgroup $L_1 \subset\nobreak L$ with the
following properties:

(1) the pair $(L_1, W)$ is geometrically equivalent to the pair
$(\GL_{a_1} \times \GL_{a_2}, \FF^{a_1} \oplus \FF^{a_2})$ for $k <
n$ and the pair $(\mathrm{S}(\mathrm{L}_{a_1} \times
\mathrm{L}_{a_2}), \FF^{a_1} \oplus \FF^{a_2})$ for $k = n$;

(2) the $(C \times L \times K_2)$-sphericity of~$Y$ is equivalent to
the sphericity of the $(C \times L_1 \times K_2)$-module $W^*
\otimes V_2$ on which $C$ and $K_2$ act as described above and $L_1$
acts on~$W^*$.

By Theorem~\ref{thm_spherical_modules_nonsimple} the indicated
module is not spherical when $k = n$, $\chi_1 = \chi_2$ and is
spherical in all other cases.

\textit{Case}~5. $K' = \Sp_{2n} \times \SL_m$, $V = \FF^{2n} \oplus
\FF^m$, $n \ge 2$, $m \ge 1$. It follows from the condition $\bl X
\br \succ \bl \Fl(1,1; V) \br$ that $\bl X \br \succcurlyeq \bl
\Fl(1,2; V) \br$. The variety $\Fl(1,2; V)$ is not $K$-spherical by
Proposition~\ref{prop_Fl(1,2;V)}, hence $X$ is not $K$-spherical
either.

\textit{Case}~6. $K' = \Sp_{2n} \times \Sp_{2m}$, $V = \FF^{2n}
\oplus \FF^{2m}$, $n \ge m \ge 2$. If $X$ is $K$-spherical then $X$
is also $(\Sp_{2n} \times \SL_{2m})$-spherical. As was shown in
Case~5, the latter is false.

The proof of Theorem~\ref{thm_non-grassmannians} and hence that of
Theorem~\ref{thm_main_theorem} is completed.
\end{proof}

\appendix

\section{Proofs of Propositions~\ref{prop_point+group_sl},
\ref{prop_point+group_sp_even}, and~\ref{prop_point+group_sp_odd}}
\label{sect_appendix}

\begin{proof}[{Proof of Proposition~\textup{\ref{prop_point+group_sl}}}]
Let $e_1, \ldots, e_n$ be a basis of~$V$. Put $W = \langle e_1,
\ldots, e_m \rangle$, and $W' = \langle e_{m+1}, \ldots, e_n
\rangle$. Denote by $L$ the subgroup of~$K$ preserving each of the
subspaces $W$ and~$W'$. It is easy to see that the point $[W] \in X$
and the group $L \subset K_{[W]}$ satisfy conditions
(\ref{4.13.2})--(\ref{4.13.4}). It remains to show that
condition~(\ref{4.13.1}) also holds.

For each $i = 1, \ldots, n$ we introduce the subspace $V_i = \langle
e_n, \ldots, e_{n-i+1} \rangle \subset \nobreak V$. The stabilizer
in~$K$ of the flag $(V_1, \ldots, V_n)$ is a Borel subgroup of~$K$,
we denote it by~$B$. It is not hard to check that the group
$B_{[W]}$ is a Borel subgroup of~$L$.

Computations show that $\dim B_{[W]} = n(n+1)/2 - m(n-m) - 1$,
whence
$$
\dim B[W] = \dim B - \dim B_{[W]} = m(n-m) = \dim X,
$$
and so the orbit $B[W]$ is open in~$X$.

Applying Proposition~\ref{prop_sufficient_(P)} to the groups~$K$,
$B$, $L$ and the point~$[W]$ we find that condition~(1) holds.
\end{proof}

In the proofs of Propositions~\ref{prop_point+group_sp_even}
and~\ref{prop_point+group_sp_odd} we shall need the following
notion.

Let $U$ be a finite-dimensional vector space with a given symplectic
form~$\Omega$ on it and let $\dim U = 2m$. A~basis $e_1, \ldots,
e_{2m}$ of~$U$ will be called \textit{standard} if the matrix of
$\Omega$ has the form
$$
\begin{pmatrix}
\begin{matrix} 0 & 1 \\ -1 & 0 \end{matrix} & & 0 \\
 & \ddots & \\
0 & & \begin{matrix} 0 & 1 \\ -1 & 0 \end{matrix}
\end{pmatrix}
$$
in this basis.

\begin{proof}[Proof of Proposition~\textup{\ref{prop_point+group_sp_even}}]
Let $\Omega$ be a symplectic form on~$V$ preserved by~$K$. We fix a
decomposition into a skew-orthogonal direct sum
$$
V = W \oplus W' \oplus R,
$$
where $\dim W = \dim W' = 2k$, $\dim R = 2n - 4k$, and the
restriction of~$\Omega$ to each of the subspaces $W, W', R$ is
nondegenerate. Let $e_1, \ldots, e_{2k}$ be a standard basis in~$W$
and let $e'_1, \ldots, e'_{2k}$ be a standard basis in~$W'$. If $n
> 2k$ (that is, $R$ is nontrivial), then we fix a linearly
independent set of vectors $r_1, \ldots, r_{n-2k}$ in $R$ that
generates a maximal isotropic subspace in~$R$. For each $i = 1,
\ldots, k$ we introduce the two-dimensional subspaces
$$
W_i = \langle e_{2i-1}, e_{2i} \rangle \subset\nobreak W \quad
\text{ and } \quad W'_i = \langle e'_{2i-1}, e'_{2i} \rangle \subset
W'.
$$

For each $i = 1, \ldots, 2k$ we put $f_i = e_i + (-1)^ie'_i$. If $n
> 2k$ then for each $j = 1, \ldots, n - 2k$ we put $f_{2k+j} =
r_j$. For each $i = 1, \ldots, n$ we introduce the subspace $F_i =
\langle f_1, \ldots, f_i \rangle \subset V$. A direct check shows
that $\Omega(f_i, f_j) = 0$ for all $i,j = 1, \ldots, n$, hence
$F_n$ is a maximal isotropic subspace of~$V$. Consequently, the
stabilizer in $K$ of the isotropic flag $(F_1, \ldots, F_n)$ is a
Borel subgroup of~$K$, we denote it by~$B$.

Put $H = K_{[W]}$. Clearly, $H$ preserves the subspace $W^\perp = W'
\oplus R$.

Put $V_i = W_i \oplus W'_i$ for $i = 1, \ldots, k$ and $V_{k+1} =
R$. Then
\begin{equation} \label{eqn_decomposition_k+1}
V = V_1 \oplus \ldots \oplus V_k \oplus V_{k+1}.
\end{equation}

We define the group $L$ to be the stabilizer in~$H$ of the flag
$(F_2, F_4, \ldots, F_{2k})$. For each $i = 1, \ldots, k$ the
$L$-invariance of the subspace~$F_i$ implies the $L$-invariance of
its projections to~$W$ and~$W^\perp$, hence both subspaces
$$
W_1 \oplus \ldots \oplus W_i \quad \text{ and } \quad W'_1 \oplus
\ldots \oplus W'_i
$$
are invariant with respect to~$L$. It follows that $L$ preserves
each of the subspaces $W_1, \ldots, W_k$, $W'_1, \ldots, W'_k$,
hence it preserves each of the subspaces $V_1, \ldots, V_k,
V_{k+1}$. At last, for each $i = 1, \ldots, k$ the $L$-invariance of
the subspaces $V_i$ and $F_i$ implies the $L$-invariance of the
subspace $V_i \cap F_i = \langle f_{2i-1}, f_{2i} \rangle$.

For each $i = 1, \ldots, k, k+1$ let $L_i$ denote the subgroup
in~$L$ consisting of all transformations acting trivially on all
summands of decomposition~(\ref{eqn_decomposition_k+1}) except
for~$V_i$. Then $L = L_1 \times \ldots \times L_k \times L_{k+1}$.

For a fixed $i \in \lbrace 1, \ldots, k \rbrace$, the $L$-invariance
of the subspace $\langle f_{2i-1}, f_{2i} \rangle$ implies that
$L_i$ diagonally acts on the direct sum $W_i \oplus W'_i$
transforming the bases $(e_{2i-1}, e_{2i})$ and $(-e'_{2i-1},
e'_{2i})$ in the same way. Consequently, $L_i \simeq \SL_2$ and the
pair $(L_i, V_i)$ is geometrically equivalent to the pair $(\SL_2,
\FF^2 \oplus \FF^2)$ with the diagonal action of $\SL_2$.

The above arguments show that the point $[W] \in X$ and the group $L
\subset K_{[W]}$ satisfy conditions (\ref{4.14.2})--(\ref{4.14.4}).
We now prove that condition~(\ref{4.14.1}) also holds.

Let us show that the group $B_{[W]} = H \cap B = L \cap B$ is a
Borel subgroup of~$L$. It is easy to see that $B_{[W]} = B_1 \times
\ldots \times B_k \times B_{k+1}$, where $B_i = B \cap L_i$ for all
$i = 1, \ldots, k, {k+1}$. For every $i = 1, \ldots, k$ the group
$B_i$ is the stabilizer in~$L_i$ of the line $\langle f_{2i-1}
\rangle$ and the group $B_{k+1}$ is the stabilizer in~$L_{k+1}$ of
the maximal isotropic flag
$$
(\langle r_1 \rangle, \langle r_1, r_2 \rangle, \ldots, \langle r_1,
r_2, \ldots, r_{n - 2k} \rangle)
$$
in~$V_{k+1}$. From this one deduces that $B_i$ is a Borel subgroup
of~$L_i$ for all $i = 1, \ldots, k, k+1$, hence $B_{[W]}$ is a Borel
subgroup of~$L$.

Computations show that $\dim B_{[W]} = (n-2k)^2 + n$, whence
$$
\dim B[W] = \dim B - \dim B_{[W]} = n^2 - (n-2k)^2 = 2k(2n-2k) =
\dim X,
$$
and so the orbit $B[W]$ is open in~$X$.

Applying Proposition~\ref{prop_sufficient_(P)} to the groups $K$,
$B$, $L$ and the point~$[W]$ we find that condition~(\ref{4.14.1})
holds.
\end{proof}

\begin{proof}[Proof of Proposition~\textup{\ref{prop_point+group_sp_odd}}]
Let $\Omega$ be a symplectic form on~$V$ preserved by~$K$. We fix a
decomposition into a skew-orthogonal direct sum
$$
V = U_0 \oplus U \oplus U' \oplus R,
$$
where $\dim U_0 = 2$, $\dim U = \dim U' = 2k$, $\dim R = 2n - 4k -
2$ and the restriction of $\Omega$ to each of the subspaces $U_0, U,
U', R$ is nondegenerate. We choose a standard basis $e_0, e'_0$
in~$U_0$, a standard basis $e_1, \ldots, e_{2k}$ in~$U$, and a
standard basis $e'_1, \ldots, e'_{2k}$ in~$U'$. If $n
> 2k + 1$ (that is, $R$ is nontrivial), then we fix a linearly
independent set of vectors $r_1, \ldots, r_{n-2k-1}$ in~$R$ that
generates a maximal isotropic subspace in~$R$. We put $W_0 =\nobreak
\langle e_0 \rangle$ and for each $i = 1, \ldots, k$ we introduce
the two-dimensional subspaces
$$
W_i = \langle e_{2i-1}, e_{2i} \rangle \subset \nobreak U \quad
\text{ and } \quad W'_i = \langle e'_{2i-1}, e'_{2i} \rangle \subset
U'.
$$

We put $f_0 = e'_0$. Next, for each $i = 1, \ldots, 2k$ we put $f_i
= e_i + (-1)^ie'_i$. At last, if $n > 2k+1$ then for each $j = 1,
\ldots, n - 2k - 1$ we put $f_{2k+j} = r_j$. For each $i = 0, 1,
\ldots, n-1$ we introduce the subspace $F_i = \langle f_0, \ldots,
f_i \rangle \subset V$. A direct check shows that $\Omega(f_i, f_j)
= 0$ for all $i,j = 0, \ldots, n-1$, hence $F_n$ is a maximal
isotropic subspace in~$V$. Consequently, the stabilizer in $K$ of
the isotropic flag $(F_0, \ldots, F_{n-1})$ is a Borel subgroup
of~$K$, we denote it by~$B$.

Put $W = W_0 \oplus U$ and $H = K_{[W]}$.

Put $V_0 = U_0$, $V_i = W_i \oplus W'_i$ for $i = 1, \ldots, k$ and
$V_{k+1} = R$. Then
\begin{equation} \label{eqn_decomposition_k+2}
V = V_0 \oplus V_1 \oplus \ldots \oplus V_k \oplus V_{k+1}.
\end{equation}

We define the group $L$ to be the stabilizer in~$H$ of the flag
$(F_0, F_2, F_4, \ldots, F_{2k})$.

Since the subspaces $W$ and $\langle e'_0 \rangle = F_0$ are
$L$-invariant, it follows that the subspace $\langle e'_0 \rangle
\oplus W = U_0 \oplus U$ is also $L$-invariant, hence so is its
skew-orthogonal complement $(U_0 \oplus U)^\perp = U' \oplus R$.
Thus $V$ admits the decomposition $V = \langle e'_0 \rangle \oplus W
\oplus (U' \oplus R)$ into a direct sum of three $L$-invariant
subspaces.

For each $i = 1, \ldots, k$ the $L$-invariance of the subspace~$F_i$
implies the $L$-invariance of its projections to~$W$ (along $\langle
e'_0 \rangle \oplus U' \oplus R$) and~$U' \oplus R$ (along $\langle
e'_0 \rangle \oplus W$), hence both subspaces
$$
W_1 \oplus \ldots \oplus W_i \quad \text{ and } \quad W'_1 \oplus
\ldots \oplus W'_i
$$
are invariant with respect to~$L$. This implies that $L$ preserves
each of the subspaces $W_1, \ldots, W_k$, $W'_1, \ldots, W'_k$,
hence it preserves $W_0$ and each of the subspaces $V_0, V_1,
\ldots, V_k, V_{k+1}$. At last, for each $i = 1, \ldots, k$ the
$L$-invariance of the subspaces $V_i$ and $F_i$ implies the
$L$-invariance of the subspace $V_i \cap F_i = \langle f_{2i-1},
f_{2i} \rangle$.

For each $i = 0, 1, \ldots, k, k+1$ let $L_i$ denote the subgroup
in~$L$ consisting of all transformations acting trivially on all
summands of decomposition~(\ref{eqn_decomposition_k+2}) except
for~$V_i$. Then $L = L_0 \times L_1 \times \ldots \times L_k \times
L_{k+1}$.

Since $L$ preserves each of the two one-dimensional subspaces $W_0 =
\langle e_0 \rangle$ and $\langle e'_0 \rangle$, it follows that
$L_0$ acts on the direct sum $\langle e_0 \rangle \oplus \langle
e'_0 \rangle$ diagonally, multiplying $e_0$ and $e'_0$ by mutually
inverse numbers. Hence $L_0 \simeq \FF^\times$ and the pair $(L_0,
V_0)$ is geometrically equivalent to the pair $(\FF^\times, \FF^1
\oplus \FF^1)$ with the action $(t, (x_1, x_2)) \mapsto (tx_1,
t^{-1}x_2)$. Next, for a fixed $i \in \lbrace 1, \ldots, k \rbrace$,
the $L$-invariance of the subspace $\langle f_{2i-1}, f_{2i}
\rangle$ implies that $L_i$ diagonally acts on $W_i \oplus W'_i$
transforming the bases $(e_{2i-1}, e_{2i})$ and $(-e'_{2i-1},
e'_{2i})$ in the same way. Consequently, $L_i \simeq \SL_2$ and the
pair $(L_i, V_i)$ is geometrically equivalent to the pair $(\SL_2,
\FF^2 \oplus \FF^2)$ with the diagonal action of $\SL_2$.

The above arguments show that the point $[W] \in X$ and the group $L
\subset K_{[W]}$ satisfy conditions (\ref{4.16.2})--(\ref{4.16.4}).
We now prove that condition~({\ref{4.16.1}}) also holds.

Let us show that the group $B_{[W]} = H \cap B = L \cap B$ is a
Borel subgroup of~$L$. It is easy to see that $B_{[W]} = B_0 \times
B_1 \times \ldots \times B_k \times B_{k+1}$, where $B_i = B \cap
L_i$ for all $i = 1, \ldots, k, {k+1}$. Evidently, $B_0 = L_0$.
Next, for every $i = 1, \ldots, k$ the group $B_i$ is the stabilizer
in~$L_i$ of the line $\langle f_{2i-1} \rangle$ and the group
$B_{k+1}$ is the stabilizer in~$L_{k+1}$ of the maximal isotropic
flag
$$
(\langle r_1 \rangle, \langle r_1, r_2 \rangle, \ldots, \langle r_1,
r_2, \ldots, r_{n - 2k - 1} \rangle)
$$
in~$V_{k+1}$. From this one deduces that $B_i$ is a Borel subgroup
of~$L_i$ for all $i = 0, 1, \ldots, k, k+1$, hence $B_{[W]}$ is a
Borel subgroup of~$L$.

Computations show that $\dim B_{[W]} = (n - 2k - 1)^2 + n$, whence
\begin{multline*}
\dim B[W] = \dim B - \dim B_{[W]} = \\ n^2 - (n-2k-1)^2 =
(2k+1)(2n-2k-1) = \dim X,
\end{multline*}
and so the orbit $B[W]$ is open in~$X$.

Applying Proposition~\ref{prop_sufficient_(P)} to the groups $K$,
$B$, $L$ and the point~$[W]$, we find that condition~(\ref{4.16.1})
holds.
\end{proof}

\end{document}